\documentclass[12pt,reqno]{article}
\usepackage[usenames]{color}
\usepackage{amssymb}
\usepackage{amsmath}
\usepackage{amsthm}
\usepackage{amsfonts}
\usepackage{amscd}
\usepackage{graphicx}
\usepackage[colorlinks=true,linkcolor=webgreen,filecolor=webbrown,citecolor=webgreen]{hyperref}
\definecolor{webgreen}{rgb}{0,.5,0}
\definecolor{webbrown}{rgb}{.6,0,0}
\usepackage{color}
\usepackage{fullpage}
\usepackage{float}
\usepackage{graphics}
\usepackage{latexsym}
\usepackage{mathrsfs}
\usepackage{epsf}
\usepackage{breakurl}
\usepackage{mathtools}

\usepackage{array}  
\usepackage{caption}

\theoremstyle{plain}
\newtheorem{theorem}{Theorem}

\newtheorem{lemma}[theorem]{Lemma}
\newtheorem{proposition}[theorem]{Proposition}

\theoremstyle{definition}
\newtheorem{definition}[theorem]{Definition}
\newtheorem{example}[theorem]{Example}
\newtheorem{conjecture}[theorem]{Conjecture}

\theoremstyle{remark}
\newtheorem{remark}[theorem]{Remark}
\newtheorem{observation}[theorem]{Observation}

\def\inv#1{\frac{1}{#1}}
\def\ie{{\rm i.e.,\/}\ }
\def\?{\,\buildrel ?\over =\,} 

\newcommand{\seqnum}[1]{\href{https://oeis.org/#1}{\rm \underline{#1}}}

\newcolumntype{L}{>{$}l<{$}} 
\newcolumntype{C}{>{$}c<{$}} 

\begin{document}


\begin{center}
\vskip 1cm{\LARGE\bf Counting Partitions by Genus:\\[8pt]
 a Compendium of Results}
\vskip 1cm
\large
Robert Coquereaux\\
Aix Marseille Universit\'e\\
Universit\'e de Toulon, CNRS, CPT\\
F-13284 Marseille Cedex 09\\
France\\
\href{mailto:email}{\tt robert.coquereaux@gmail.com} \\
\ \\
 Jean-Bernard Zuber\\
 Sorbonne Universit\'e, CNRS \\
 Laboratoire de Physique Th\'eorique et Hautes Energies, LPTHE \\
 F-75252 Paris \\
 France\\
 \href{mailto:email}{\tt zuber@lpthe.jussieu.fr}
\end{center}

\vskip .2 in
\begin{abstract}
We study the enumeration of set partitions, according to their length, number of parts, cyclic type, and genus.
We introduce genus-dependent Bell,  Stirling numbers, and Fa\`a di Bruno coefficients. 
Besides attempting to summarize what is already known on the subject, we obtain new generic results (in particular for partitions into two parts, for arbitrary genus),
 and present computer generated new data extending the number of terms known for sequences or families of such coefficients; 
  this also leads to new conjectures.
\end{abstract}

\section{Introduction}
This is the second paper in a series devoted to the combinatorics of set partitions and 
their enumeration according to their genus. In a previous 
paper~\cite{Z23}, functional equations were written between generating functions (G.F.) of 
partitions, enabling one to count partitions in genus 0, 1 and 2. In the present paper, which is
completely independent, our goal is different. We wish to collect as much  data as possible
on that combinatorics. Accordingly, our paper gathers classical and known results as 
well as new data, obtained by computer ``brute force'' calculations and a few exact new
results. In many cases, these data suggest conjectures and extrapolations, that we mark 
with the sign $\?$. 

This endeavor has benefited in a tremendous way from the existence of the 
{\it On-Line Encyclopedia of Integer Sequences} (OEIS)~\cite{OEIS}. Several unexpected connections and identifications have been made
possible thanks to this irreplaceable and unique source.

Our paper is organized as follows. In  Section \ref{sect2}, we recall some basic definitions: 
 total numbers of partitions are given by  Bell  numbers, and  by Stirling numbers when the number of parts
 is fixed. The key notion of genus is also recalled.  
In Section \ref{sect3},  the Bell numbers are refined by fixing the genus of partitions, and 
by including partitions with or without singletons. 
Explicit expressions for their G.F. are given in genus 0 to 2, conjectured in genus 3, and the general form is discussed
 in higher genus. The same steps are repeated in Section~\ref{sect4} for the Stirling numbers,
with again exact or conjectured results for their counting and G.F.
The rest of the paper is devoted to the counting of partitions of given cyclic type.
In Section~\ref{sect5}, we review three families of partitions for which this counting is generically known:
the famous non-crossing partitions (\ie of genus 0); {the partitions of genus 1 and 2, 
 \cite{Z23};} the partitions into pairs, \ie of type $[2^k]$,
and arbitrary genus; and the partitions into two parts, for which we obtain a result  in arbitrary genus, which is new, 
to the best of our knowledge. Section~\ref{sect6} gathers data on various types of partitions
 for which we have only partial results and conjectures: types $[p^k]$ for 
varying $p$ or $k$ and three-part partitions.
Finally, the tables in the appendix  contain the number of partitions of the set $\{1,\ldots,n\}$ of arbitrary 
genus up to $n=15$. 

To put the present work in perspective, let us recall some background. For a long time, the census of partitions according to their genus 
has been confined to two 
particular cases. 

On the one hand, non-crossing partitions, \ie of genus 0 in the current approach, and  of arbitrary type,
 have been enumerated by Kreweras \cite{Krew}. His result reappeared in the context of large size matrix integrals and their ``planar''  limit \cite{BIPZ}:
 there, connected correlation functions (also known as cumulants) were in one-to-one correspondence with  non-crossing partitions.
The  generating function of the latter
was shown to satisfy a remarkably simple functional equation, from which Kreweras 
 result followed by use of Lagrange inversion formula. That functional equation in turn received a simple diagrammatic interpretation by Cvitanovic  \cite{Cvitanovic}.
 Non-crossing cumulants associated with non-crossing partitions then appeared in the framework of free probability in the work of Speicher \cite{Speicher}.

 On the other hand, and in a totally independent vein, the special class of partitions of a set of even cardinality into pairs and 
 of arbitrary genus was treated by Walsh and Lehman~\cite{WL1}, 
 and then by Harer-Zagier.  See more references in Section \ref{sec:HZ}. 
 It is only recently that  general partitions of higher genus were reconsidered by Cori and Hetyei \cite{CoriH13, CoriH17}, and results on
 the counting of partitions with a given number of parts and genus 1 and 2 were obtained. Their method relies on a reduction of diagrams
 to simpler ``irreducible''  ones, with the same genus,  which they proved to be in finite number for given genus. They were able 
 to list these irreducible diagrams for genus 1 and 2. 
 This enabled them to determine the G.F. of partitions with a given number of parts, for genus 1 and 2.  See (\ref{CHg1})-(\ref{CHresult}). 
 This method was pushed one step further
 in \cite{Z23}, using the same  reduction to irreducible diagrams, followed by a reconstruction of all partitions 
 from those irreducible diagrams by use of functional 
 identities generalizing Cvitanovic's one in higher genus. That method is limited in genus by the increasing complexity in 
 first listing all irreducible diagrams in genus greater or equal to 3 and then in ``dressing''  them. 
 
 As this short review shows, there is ample room for 
 further progress and we hope the present work will stimulate the curiosity and imagination of some readers.

\section{Bell, Stirling, and Fa\`a di Bruno numbers}
\label{sect2}
\subsection{Equivalence relations on a set with \texorpdfstring{$n$}{n} elements}

Any equivalence relation on a set is specified by a partition $\alpha$ of this set (or set-partition, for short), 
{and conversely, an equivalence relation defines  a partition.} 

The number of equivalence relations on a set with $n$ elements is given by the Bell numbers, which obey the recurrence:
\begin{equation} B_{n+1} = \sum_{p=0}^n  \binom {n} {p}  B_p, \; \text{with} \; B_1=1 ,\end{equation}

hence \begin{equation} B_n=\inv{e} \sum_{\ell=0}^\infty \frac{\ell^n}{\ell!}, \qquad \text{OEIS sequence \seqnum{A000110}} 
.\end{equation}

The exponential generating function of the Bell numbers is
\begin{equation} {\mathcal B}(x) = \sum_{n=0}^\infty \frac{B_n}{n!} x^n = e^{e^x-1}. \end{equation}

\subsection{Equivalence relations on a set with \texorpdfstring{$n$}{n} elements, with \texorpdfstring{$k$}{k} equivalence classes}

One may impose that the equivalence relations have a given number, say $k$, of equivalence classes,
{\ie that the partition has $k$ parts}.

The number of such relations is given by the Stirling number of the second kind, $S_{n,k}$, which obeys the recurrence relation: 
\begin{equation} S_{n, k} = k \, S_{n-1, k} + S_{n-1, k-1},\; n > 1, \; \text{with} \;   S_{1, k} = 0, \, k > 1, \;\text{and}\;  S_{1, 1} = 1.\end{equation}

An explicit form is
\begin{equation} S_{n, k} = \frac{1}{k!} \, \sum_{s=0}^k (-1)^{k-s}  {\binom {k}  {s}} \, s^n,   \qquad {\text{OEIS sequence \seqnum{A008277}}}.\end{equation}

The exponential generating function of the Stirling numbers $S_{n,k}$  is 
\begin{equation} {\mathcal S}(x,y) = e^{y \, (e^x-1)} \, \end{equation}
{with the following ``sum rule'' expressing their sum over $k$ as a Bell number,}
\begin{equation} B_n = \sum_{k=0}^n S_{n,k}.\end{equation}

\subsection{Set with $n$ elements, with $k$ equivalence classes of specified cardinalities}

One may further impose that the chosen equivalences classes have specified cardinalities.

The set of  cardinalities of the classes of the  equivalence relation defined by the set-partition $\alpha$
 determines a partition $[\alpha]\vdash n$ of the {\it integer} $n$,  called the {\it type} of the partition.
It is usual to denote this integer partition as follows:  $[\alpha]=[1^{\alpha_1}, \ldots, n^{\alpha_n}]$.
 It can be represented as a Ferrers diagram or as a Young diagram.
The number of equivalence relations on a set with $n$ elements having equivalent classes with cardinalities specified by $[\alpha]$ will be denoted by $C_{n, [\alpha]}$.
These numbers are sometimes called the Fa\`a di Bruno coefficients.

  \begin{equation}\label{faadiBruno}C_{n,[\alpha]}=\frac{n!}{\prod_{\ell=1}^n \alpha_\ell!  (\ell!)^{\alpha_\ell}}. \end{equation}

Sum rule: calling 
{$|\alpha|=\sum_\ell \alpha_\ell$ } the number of parts of the integer partition $[\alpha]$, we have obviously
\begin{equation} \sum_{\substack { [\alpha] \\  |\alpha|=k} }\,  C_{n, [\alpha]} = S_{n,k}. \end{equation}

 \subsection{Genus of partitions on a cyclically ordered set} 
 \label{sectdefgenus}
If the underlying set of $n$ elements is totally ordered (for definiteness one may take it as $\{1,2,3,\ldots,n\}$), or if it is cyclically ordered, one may introduce a new structure, finer than the ones already considered, by determining the genus of set partitions (a non-negative integer).
 With $\alpha$ a  partition of $\{1,2,3,\ldots,n\}$, we associate  a {\it permutation}  $\tau$ of ${\mathcal S}_n$: its cycles 
are the parts of $\alpha$, {\it with the important constraint that their elements are  in  increasing order}. We also consider the cyclic permutation
$ \sigma:= (1,2, \ldots,n)$. Then following \cite{Jacques, WL1, WL2}, the genus $g(\alpha)$ is defined by
\begin{equation}\label{defgenus}  n+2-2g=  \#\mathrm{cy}(\tau)+\#\mathrm{cy}(\sigma) + \#\mathrm{cy}(\sigma\circ \tau^{-1})\,\end{equation}
or in the present case, 
\begin{equation}\label{genus} -2g =|\alpha| -1 - n  + \#\mathrm{cy}(\sigma\circ \tau^{-1}) \end{equation}
since here $\#\mathrm{cy}(\sigma) =1$ and  $\#\mathrm{cy}(\tau)= \sum \alpha_\ell=|\alpha|$.

For the reader's convenience, we recall briefly the diagrammatic 
representation of partitions, that makes  the topological 
interpretation of  the definition (\ref{genus}) more transparent.

\begin{figure}\begin{center}
{\includegraphics[width=.9\textwidth]{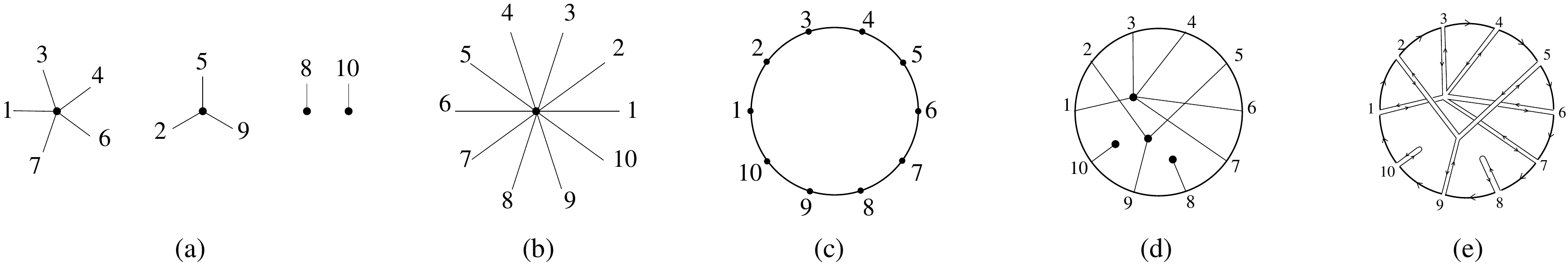}}
\end{center} 
\caption{The  partition $(\{1,3,4,6,7\},\{2,5,9\},\{8\},\{10\})$ of $\{1,\ldots,10\}$.   (a) the four $\ell$-vertices;  (b) and (c): two equivalent representations of the special 10-vertex; (d) a contribution to $C^{(g)}_{10, [1^2\,3 \, 5]}$; (e)  the double line (fat graph) version of (d), with  three faces and thus genus $g=2$.}
\label{Fig1}
\end{figure}

To a given partition, we may also attach a map: it has  $\alpha_\ell$  $\ell$-valent vertices, in short {\it $\ell$-vertices}  
whose edges are numbered clockwise by the elements of the partition, (see Fig.~\ref{Fig1}a),  and
a special $n$-vertex, with its $n$ edges numbered {\it anti}-clockwise from 1 to $n$. See.~\ref{Fig1}b. 
Edges are connected pairwise by matching their indices.
 Two maps are regarded as topologically equivalent if they encode the same partition. In fact
it is topologically equivalent and more handy to attach  $n$ points {\it clockwise}  on a circle, and to connect them pairwise by arcs of the circle. See Fig.~\ref{Fig1}b. 
Now the permutation $\sigma$ describes the connectivity of the $n$ points on the circle, while $\tau$ describes how these points are
connected through the $\ell$-vertices  (drawn inside the disk). It is readily seen that the permutation $\sigma\circ\tau^{-1}$ describes the circuits bounding clockwise
the faces of the map. This is even more clearly seen if one adopts a double line notation for each edge \cite{tH}, thus transforming 
the map into a ``fat graph''. See Fig.~\ref{Fig1}e.
 Consecutive entries of the cycles of $\sigma\circ\tau^{-1}$ label the tails of those oriented edges that are drawn going  inward: in this way one can read the entries of $(\{1, 8, 9, 6, 5, 3, 2, 10\},\{4\},\{7\})$ by following the three circuits of Fig.~\ref{Fig1}e.
Thus the number of cycles of $\sigma\circ\tau^{-1}$  is the number $f$ of faces of the map. 
Since each face is homeomorphic to a disk, gluing a disk to each face transforms the map into a closed Riemann surface, to which we may apply Euler's formula

\begin{equation}\label{Euler}2 -2g = \#(\mathrm{vertices})-\#(\mathrm{edges})+\#(\mathrm{faces})= 1+\sum_\ell \alpha_\ell - n +f\end{equation}
with $f=\#\mathrm{cy}(\sigma\circ \tau^{-1})$,  and we have reproduced (\ref{genus}). 
{In other words, $g$ is the minimal genus of the surface on which the map  induced by  the partition may be
drawn without crossings.}
 
Since $f$ is a positive integer,
and letting $k=|\alpha|$ denote the number of parts of the integer partition $[\alpha]$, we have
\begin{equation} \label{gversusface} g =\frac{n-k+1-f}{2}\le \frac{(n-k)}{2}. \end{equation}
Also note that one-part partitions ($k=1$) necessarily have genus 0. Hence for $g>0$, $k\ge 2$ and
\begin{equation}\label{boundonng}n\ge 2g+k \ge 2g+2.\end{equation}

\begin{remark}
Assuming the existence of an order on the underlying set is not really a restriction, since one can always choose one.
Each family of set-partitions (or equivalence relations) previously considered will be itself decomposed according to the genus, and we shall introduce notation 
$B_{n}^{(g)}$, $S_{n, k}^{(g)}$ and $C_{n, [\alpha]}^{(g)}$, with
\begin{equation} \sum_g B_{n}^{(g)}=B_{n}, \; \sum_g S_{n, k}^{(g)}=S_{n, k}, \ \text{and} \ \sum_g C_{n, [\alpha]}^{(g)}  =C_{n, [\alpha]}. \end{equation}
\end{remark}

\begin{remark}
Let us emphasize again that the maps associated with set partitions are subject to the important constraint that 
their vertices (= parts) are ordered. Accordingly, the edges connecting each $\ell$-vertex to the $n$-vertex cannot cross
one another, thus respecting their original cyclicity  and ordering. Only crossings of edges originating from distinct vertices are allowed. 
It is that constraint that makes the census of partitions difficult.
\end{remark}

\begin{remark}
Tables and conjectures follow from computer calculations, using Mathematica.
From a computational point of view, tables giving the number of set partitions for a set of $n$ elements with given genus, possibly obeying some constraints, can be obtained from the following simple algorithm:
\begin{itemize}
\item[(1)] Generate all partitions of this set. 
\item[(2)] Compute their genus, using formula (\ref{defgenus}).
\item[(3)] Select those partitions that obey the chosen particular constraints (given number of parts, specific cyclic type, absence of singletons, \dots). 
\end{itemize}
This method is simple enough to implement, and we could follow the above steps using set partitions generated by Mathematica, for ``small'' values of $n$; we also used a predefined command that can generate lists containing all partitions with given number of blocks (parts).

Unfortunately,  the large number of generated partitions makes the method intractable when $n$ increases. 

A first simplification consists in focusing on partitions without singletons, since the others (or their number) can be simply obtained from the former. 

However, when $n$  is too large, typically $n >12$, the amount of RAM required to hold all these partitions exhausts the possibilities of a typical laptop, even if one restricts his attention to partitions without singletons; moreover, using virtual memory techniques slows down considerably the calculations. In order to handle higher values of $n$ we had to write programs (in Mathematica) that build partitions sequentially, \ie one at a time rather than in long lists, sometimes generating only those of a given cyclic type, then saving those partitions, or only their number, only if they obey some chosen criteria (according to their genus for instance). 

Another method that we used for such higher values of $n$ was to save large lists of partitions into external files and subsequently use stream processing techniques (manipulating pointers) to process each record one at a time. 
\end{remark}

It is  important to notice that, unfortunately, we could not devise an efficient algorithm that would have allowed us to generate directly all the partitions that have a specific genus.

In the following, we shall be using both representations, by pairs of permutations or by fat graphs, in turn or in parallel.

\subsection{Partitions with no singletons}
If a set partition has no singleton, its associated equivalence relation is such that no element is isolated. 
Equivalently, each part of the partition contains at least $2$ elements.
By adding the constraint that the families of partitions considered previously should have no singleton one can define ``associated''\footnote{More generally, one could introduce $s$-associated Bell or Stirling numbers by imposing that each part contains at least $s$ elements but in the present paper we consider only the case  $s=2$. 
See \cite{Caicedo-etal}.} Bell numbers $\widehat B_n$, and ``associated'' Stirling numbers (of the second kind)  $\widehat S_{n,k}$. Obviously,
\begin{equation} \sum_{k=1}^n \widehat S_{n,k} = \widehat B_n.\end{equation}
Since the notation $C_{n, [\alpha]}$ already incorporates the partition type, there is no ``hat'' version of the Fa\`a di Bruno coefficients: either $[\alpha]$ contains singletons, or it does not.

\subsubsection{Associated Bell number $\widehat B_n$} See OEIS sequence \seqnum{A000296}.

Their exponential generating function is 
\begin{equation}\widehat {\mathcal B}(x) = e^{e^x - x - 1}.\end{equation}
Let us mention the following identities:
\begin{align} B_n &= \widehat B_n + \widehat B_{n+1}  \\
 \widehat B_n&= \sum_{j=0}^{n-2} (-1)^j \, B_{n-1-j}. \end{align}
 {Using generating functions, the first relation is a  consequence of the following equality~:
$ \frac{d}{dz} \exp (e^z-z-1)=\exp (e^z-1)-\exp (e^z-z-1)$.
 Both relations are mentioned in the OEIS in sequence \seqnum{A000296}.}

\subsubsection{Associated Stirling numbers of the second kind $ \widehat S_{n,k}$} See OEIS sequence \seqnum{A008299} (also see \seqnum{A134991} where they are called Ward numbers). 
 \begin{equation} \widehat S_{n,k} = \sum_{\ell=0}^{k} (-1)^\ell  \, {\binom {n} {\ell}} \, S_{n-\ell,k-\ell} \label{WardFromStirling} \end{equation} 
Conversely,
\begin{equation} S_{n,k} = \sum_{\ell = 0}^{k-1} \, {\binom {n}  {\ell}}  \, \widehat S_{n-\ell,k-\ell} \label{StirlingFromWard} \end{equation}
Their exponential generating function is
\begin{equation} \widehat {\mathcal S}(x,y) = e^{y \, (e^x-x-1)} .\end{equation}
They can be expressed {(see OEIS sequence \seqnum{A008299})} in terms of the second-order Eulerian numbers $E^{(2)}$, as defined in \cite[p.~256]{GKP}, by 
\begin{equation} \widehat S_{n,k} = \sum_{\ell = 0}^{n-k} \, \binom{\ell} {n-2k} \, E^{(2)}_{n-k, n-k-\ell} \label{WardFromEulerian}  \end{equation}
The $E^{(2)}$'s can themselves be expressed in terms of Stirling numbers of the second kind {(see OEIS sequence \seqnum{A340556})} by 
\begin{equation} E^{(2)}_{n,k} = \sum_{j=0}^k \, (-1)^{k - j}  {\binom {2n+1}  {k-j}}  \ S_{n + j, j}  \label{EulerianFromStirling}  \end{equation}
From (\ref{WardFromEulerian}) and (\ref{EulerianFromStirling}) one can recover (\ref{WardFromStirling}).

\bigskip

Notice that
\begin{equation} C_{n,[ 1^r,\, \alpha^\prime]}= {\binom {n}  {r}} \, C_{n-r,[\alpha^\prime]},\end{equation}
 where $\alpha'$ has no singleton.

One can also impose a genus restriction on the partitions without singletons, 
and as singletons do not affect the genus, one is therefore led to consider numbers  $\widehat B_n^{(g)}$ and $\widehat S_{n,k}^{(g)}$ with, of course $\sum_g \widehat B_n^{(g)}=\widehat B_n$ and $\sum_g \widehat S_{n,k}^{(g)} = \widehat S_{n,k}$. 
Moreover
 \begin{equation}\label{with/osingle} C_{n,[ 1^r,\, \alpha^\prime]}^{(g)}= {\binom {n}  {r}} \, C_{n-r,[\alpha^\prime]}^{(g)}. \end{equation}
We shall return to these sequences in the next section.

\section{Genus-dependent Bell numbers $B_{n}^{(g)}$}
\label{sect3}
\subsection{Unconstrained partitions: basic numbers $B_{n}^{(g)}$}

\subsubsection{Genus $0$} Known as  Catalan numbers.
See OEIS sequence \seqnum{A000108}.
\begin{align}\label{Catalan}B_{n}^{(0)}  =  {\mathcal C}_n &:= \frac{1}{(n+1)!} \, \frac{(2 n)!}{n!}\\
  &= \nonumber \{1,2,5,14,42, 132, 429, 1430, 4862, 16796, \ldots\}\end{align}
The ordinary G.F. is  
\begin{equation}B^{(0)}(x)=\frac{1-\sqrt{1-4x}}{2x} .
\end{equation}

\subsubsection{Genus $1$} See OEIS sequence \seqnum{A002802}.  We have
\begin{equation} B_{n}^{(1)} =\frac{1}{2^4\, 3} \, \frac{1}{(2 n-3) (2 n-1)} \, \frac{1}{ (n-4)!} \, \frac{(2 n)!}{n!}, \end{equation}
The ordinary G.F. is
\begin{equation}  {B^{(1)}(x)=}\frac{x^4}{(1-4 x)^{5/2}} . \qquad \text{See \cite{CoriH17}} .\end{equation}

\subsubsection{Genus  $2$} The first few terms are listed in the OEIS as
sequence \seqnum{A297179}. 
The next formula  seems to be new; it is obtained by summing  $S_{n, k}^{(2)}$, see (\ref{CHresult}) given below, over $k$, the number of parts.
\begin{equation} B_{n}^{(2)}=\frac{1}{2^9\, 3^2\, 5} \, \frac{\left(5 n^3-39 n^2+88 n-84\right)}{(2 n-7) (2 n-5) (2 n-3) (2 n-1)}\, \frac{1}{ (n-6)!} \, \frac{(2 n)!}{n!},\end{equation}
The ordinary G.F. is
\begin{equation} {B^{(2)}(x)=}\frac{x^6 \left(1+ 6 x-19 x^2+ 21 x^3\right)}{(1-4 x)^{11/2}}.  \qquad \text{See \cite{CoriH17}} .\end{equation}

\subsubsection{Genus $3$}  We conjecture that
\begin{equation} 
\resizebox{.91\linewidth}{!} {$B_{n}^{(3)}\? \frac{1}{2^{13}\, 3^4\, 5\, 7} \, \frac{\left(35 n^6-819 n^5+7589 n^4-36009 n^3+93464 n^2-129060 n+95040\right)}{(2 n-11) (2 n-9) (2 n-7) (2 n-5) (2 n-3) (2 n-1)}\, \frac{1}{ (n-8)!} \, \frac{(2 n)!}{n!}$},
\end{equation}

The ordinary G.F. is 
\begin{equation}
\resizebox{.60\linewidth}{!} {$B^{(3)}(x)\? \frac{x^8 \left(1+ 60 x -66 x^2-130 x^3+1065 x^4 -2262 x^5+1738 x^6\right)}{(1-4 x)^{17/2}}$},
\end{equation}

This suggests for any $g>0$ the following Ansatz for the ordinary G.F.
\begin{equation}\label{AnsatzBgx}  B^{(g)}(x) \? \frac{x^{2g+2}P^{(g)}(x)}{(1-4 x)^{(6g-1)/2}},   \end{equation}
with an overall power of $x$ dictated by (\ref{boundonng}) and a polynomial $P^{(g)}$ of degree $3(g-1)$.

{We shall see in the sequel a repeated appearance of formulae of that type, in particular with the universal
``critical exponent'' $(6g-1)/2$ in the denominator.}
\smallskip

For genus $g\ge 4$, we have incomplete results that  corroborate this Ansatz.

\subsubsection{Genus $4$} The formula below is conjectured, and
one should compute $B_{n}^{(4)}$ for $n = 16, 17$, $18, 19$ to determine all the coefficients $a_i$ 
\begin{equation} B_{n}^{(4)}= \{1, 352, 19261, 541541, 10571561, 162718556\} \  \text{for\ } n=10,\ldots 15.\end{equation}
The ordinary G.F. should be
\begin{equation*}
\resizebox{.91\linewidth}{!} {$B^{(4)}(x) \?   \frac{x^{10} \left(1 +306 x+4035 x^2 -16669 x^3 +63735 x^4 -136164 x^5   +{a_6} x^6+{a_7} x^7+{a_8} x^8+{a_9} x^9\right)}{(1-4 x)^{23/2}}$}.
 \end{equation*}

 Similarly, we   propose for $g=5,6$ that
  \begin{align} \nonumber B^{(5)}(x) &\? \frac{x^{12}(1+1320 x  +75068 x^2 + 218300 x^3 +\cdots )}{(1 - 4 x)^{29/2}}\\
    B^{(6)}(x) &\? \frac{x^{14}(1+5406x  +\cdots )}{(1 - 4 x)^{35/2}}.
\end{align}
The first nontrivial coefficient in the numerator of $ B^{(g)}(x)$, $g>0$, 
  appears to be always divisible by 6:
$6{\times}\{1,10,51, 220,901,\ldots\}$, for $g=2, 3,\ldots$,  
{and we conjecture that this sequence is given by
$$\frac{(d(g)+8 g+2) (6 g-2)!}{C(g) (3 g-1)!}-2 (6 g-1)$$
in terms of  $C(g)$:  $C={12, 30240, 518918400, 28158588057600, 3497296636753920000, \ldots}$,  \\
and $d(g)$: $d={0, 10, 68, 318, 1336, 5426,\ldots}$,
given by 
\begin{align} \label{Coef} C(g)&=3  \times 2^{2g-1}\frac{(2g)!}{g!}  \frac{(6g-5)!!}{(2g-3)!!}{= 12 (2g-1)\frac{(6g-5)!}{(3g-3)!}}\\
\label{dcol2}d(g) &= \frac{1}{3} ( 4^{1+g}-1 -3(6g-1)).\end{align}}

 Warning: the sequence $B_{n}^{(0)}$, shifted in such a way that it starts with $1$ for $n = 1$, (resp., $B_{n}^{(1)}$, shifted in such a way that it starts with $1$ for $n = 3$) gives also the number of rooted bi-colored unicellular maps of genus $0$ (resp., of genus $1$) on $n$ edges.  
 However,  the counting of partitions differs from that of maps, as already discussed, and we observe that  the above coincidence fails at genus $2$ and above. 
  Rooted  bi-colored unicellular maps are studied by Goupil {and Schaeffer} \cite{GoupilSch}.

\subsection{Partitions with no singletons: associated numbers $\widehat B_{n}^{(g)}$}

For all $g$ one has the recurrence 
\begin{equation} \widehat B_{n}^{(g)} = B_{n}^{(g)} - \sum_{s=1}^n \,  {\binom {n}  {s}} \, \widehat B_{n-s}^{(g)}  \; \text{with} \; \widehat B_{n}^{(g)} = 0\; \text{for}\, n<2g+2, \text{and}\,  \widehat B_{2g+2}^{(g)}  =1. \end{equation}

\subsubsection{Genus $0$}   See OEIS sequence \seqnum{A005043} (Riordan numbers).
\begin{equation} 
\widehat B_{n}^{(0)} = \sum_{j=0}^n \, (-1)^j \, {\binom{n} {j}}  {\binom{j}  {\lfloor j/2 \rfloor}}.   
\end{equation}
The ordinary G.F. is 
\begin{equation}\frac{2}{1+x+\sqrt{(1 - 3 x) (1 + x)}}={\frac{1-\sqrt{\frac{1-3x}{1+x}}}{2x}},
\end{equation}  
with coefficients
$${0, 1, 1, 3, 6, 15, 36, 91, 232, 603, 1585, 4213, 11298, 30537, 83097, 227475,\ldots}.$$
 
\bigskip

For genus $g > 0$, we again have a general Ansatz for the ordinary G.F. of  $\widehat B_{n}^{(g)}$
\begin{equation} \widehat B^{(g)}(x) \? \frac{x^{2 (g + 1)} (1+x)^{g-1} \, \widehat P^{(g)}(x)} {{\Delta(x)}^{\frac{(6 g - 1)}{2}}}\end{equation}
where
 $\Delta(x) = (1 - 3 x) (1 + x)$ {is the discriminant of the algebraic equation satisfied by  
$\widehat B^{(0)}(x)$, namely  $ \widehat B^{(0)}(x)   =1+(x   \widehat B^{(0)}(x))^2/(1- x\widehat B^{(0)}(x) )$
(see \cite{Z23}),  and $\widehat P^{(g)}$  is }a polynomial  of degree $3(g-1)$.  See below.

\subsubsection{Genus $1$}  See OEIS sequence \seqnum{A245551}. 
\begin{equation} \label{HZg1} \widehat B_{n}^{(1)} = \sum_{\ell=0}^{n-4}\,  \frac {(-1)^{n - \ell} \, 3^{\ell - 2}}{2^{n-4}}  \frac{(2 \ell + 3)\text {!!}  (2n -2 \ell - 5)\text {!!}} {\ell! (-\ell + n - 4)!} ,
\quad   \text{G.F.}\, \widehat B^{(1)}(x) \,\text{with}\,  \widehat P^{(1)}(x)=1,\end{equation}
\[ \scalebox{0.85}{0, 0, 0, 1, 5, 25, 105, 420, 1596, 5880, 21120, 74415, 258115, 883883, 2994355, 10051860,\ldots} .\]

\subsubsection {Genus $2$} 
\[ \widehat B_{n}^{(2)} = \scalebox{0.85}{0, 0, 0, 0, 0, 1, 21, 203, 1512, 9513, 53592, 278355, 1359072, 6318312, 28227199, 122005884,\ldots} \]
\begin{equation} \label{HZg2} \text{G.F.}\, \widehat B^{(2)}(x) \,\text{with}\, \widehat P^{(2)}(x)=(1 + 9 x - 4 x^2 + 9 x^3). \qquad \text{See \cite{Z23}} .\end{equation}

\subsubsection{Genus $3$}
\[ \widehat B_{n}^{(3)} = \scalebox{0.85}{0, 0, 0, 0, 0, 0, 0, 1, 85, 1725, 21615, 208230, 1685112, 12028588, 78029380, 469278810, \ldots} \]
\begin{equation} \text{G.F.}\, \widehat B^{(3)}(x) \,\text{with}\,  \widehat P^{(3)}(x) \?
(1 + 66 x + 249 x^2 + 226 x^3 + 894 x^4 - 480 x^5 + 406 x^6).\end{equation}

\subsubsection{Genus $4$} 
\[ \widehat B_{n}^{(4)} =  \scalebox{0.85}{  0, 0, 0, 0, 0, 0, 0, 1, 341, 15103, 318318,   4615611 ,      52720668  , \ldots }\]
\[\text{G.F.}\, \, \widehat B^{(4)}(x) \,\text{with}\, P^{(4)}(x)\?  1 + 315 x + 6519 x^2 + 20228 x^3 + 65718 x^4 + 95247 x^5+\cdots\] 
with 4 terms missing.

\section{Genus-dependent Stirling numbers $S_{n, k}^{(g)}$}
\label{sect4}
\subsection{Partitions with $k$ parts: the numbers $S_{n, k}^{(g)}$}

\subsubsection{Genus $0$} Known as the Narayana numbers. See OEIS sequence \seqnum{A001263}.
\begin{equation} S_{n, k}^{(0)} = \frac{1}{n} \, {\binom {n}  {k}} \, {\binom{n}  {k-1}} =  \frac{1}{k} \, {\binom{n-1}  {k-1}} \, {\binom{n}  {k-1}} .\end{equation}
Their two-variable G.F. is
\begin{equation} S^{(0)}(x,y) = \frac{1 + x- x y - \sqrt{ (1 + x - x y)^2-4x}}{2 x} .\end{equation}

\subsubsection{Genus $1$}Conjectured by Yip \cite{Yip}, proved by Cori and Hetyei~\cite{CoriH13}. 
\begin{equation}\label{CHg1} S_{n, k}^{(1)} = \frac{1}{6} \, {\binom{n}  {2}} \, {\binom{n-2}  {k-2}} \,  {\binom{n-2}  {k}} = \frac{1}{6} \, {\binom{k}  {2}} \, {\binom{n}  {k}} \, {\binom{n-2}  {k}} ,\end{equation}
\begin{equation} \text{G.F.}\qquad S^{(1)}(x,y) =   \frac{x^4 y^2}{( (1 +x -x y)^2 -4x)^{5/2}}.\end{equation}

\subsubsection{Genus $2$} Obtained by Cori and Hetyei. See OEIS sequence \seqnum{A297178}.
\begin{equation} \begin{split}
S_{n, k}^{(2)} =& 8 \gamma[n - 10, k - 6] - 
 4 \gamma[n - 10, k - 5] - 
 15 \gamma[n - 10, k - 4] + 
 10 \gamma[n - 10, k - 3]  \\ & + 
 \gamma[n - 10, k - 2]  - 
 4 \gamma[n - 9, k - 5] + 
 39 \gamma[n - 9, k - 4] - 
 10 \gamma[n - 9, k - 3] \\ & - 
 4 \gamma[n - 9, k - 2] - 
 15 \gamma[n - 8, k - 4]  - 
 10 \gamma[n - 8, k - 3] + 
 6 \gamma[n - 8, k - 2]\\ &  - 
 4 \gamma[n - 7, k - 2] + 
 10 \gamma[n - 7, k - 3] + \gamma[n - 6, k - 2]
 \end{split} \label{CHresult}
  \end{equation}
 where
 \begin{equation}\gamma[n,k]= \frac{  \binom{n+10} {5}   \binom {n+5}  {k}   \binom{n+5}  {n-k}} { \binom{10} {5}}.\end{equation}

 We now introduce some new notation:
\begin{definition}
Given two 2-indexed sequences $a_{st}$ and $b_{st}$,  let
 \begin{equation}\label{def*} a*_p b = \sum_{0\le s \le t \le p}   a_{st}  b_{st}. \end{equation}
\end{definition}

 Using this notation, the result for $S_{n, k}^{(2)}$ may be simplified in the form
 \begin{equation} 
S_{n,k}^{(2)} = \inv{30240} \; \chi^{(2)} *_{4} \Big(\frac{\Gamma(n-t) \Gamma(n-t+5)}{ \Gamma(k-s-1)\Gamma(k-s+4)\Gamma(n-k +s-t-3)\Gamma(n-k+s-t+2)}\Big)
\label{simpCHresult}
\end{equation}
with a triangular array of constants $\chi^{(2)}$ given by
\begin{equation}\label{chi2} \chi^{(2)}{(t,s)}=  \begin{matrix}1&&&&\cr- 4&10&&&\cr 6&-10&-15&& \cr -4&-10&39&-4&\cr 1 &10&-15&-4&8\cr \end{matrix}.\end{equation}

Then the G.F. is
 \begin{align} \label{FGxy2} 
  S^{(2)}(x,y) &= \frac{x^6 y^2\, p^{(2)}(x,y)}{( (1 +x -x y)^2 -4x)^{11/2}} \\
\nonumber p^{(2)}(x,y) &= {\sum_{\substack{0\le t\le 4 \\ 0\le s \le t}}  \chi^{(2)}(t,s)  x^t y^s}\\
 \nonumber  &=1- x(4-10 y) + x^2(6-10y-15y^2) -x^3(4+10y-39y^2+4y^3) \\
  &\qquad+x^4(1+10y-15y^2-4y^3+8y^4)
\end{align}
as first derived by Cori and Hetyei~\cite{CoriH17}.

\subsubsection{Observations and conjectures}

For $k=2$, we have an expression that follows from an exact formula for the two-part partitions of arbitrary genus.  See below Section~\ref{2part-partitions}:
\begin{equation}\label{Sfork=2}  S_{n, 2}^{(g)}= \binom {n} {2g+2}={ \inv{\binom{2g+2}{ g+1}}\binom{n}{ g+1} \binom{n-g-1}{ g+1}}. \end{equation}

\begin{observation} All known data for $k=3$ and $n\le 15$, $g\le 6$  are consistent with 
\begin{align}\label{ansatzk=3}  S_{n, 3}^{(g)} \, {\?} 
  \frac{4^{g+1} - 1}{3} \frac{ (n-g-1)}{g+2}\binom{n}{ 2g+3}
\,= \,  \frac{4^{g+1}- 1}{3  \binom{2g+3}{ g+1}} \binom{n}{ g+2}\binom{n-g-1}{ g+2}
. \end{align}  
In particular
\begin{equation} S_{n, 3}^{(2)} {\?}  \frac{21}{4} \binom{n}{ 7} (n - 3)=\frac{3}{5} \binom{n}{ 4} \binom{n-3}{ 4}\; 
\text{and}\;\; S_{n, 3}^{(3)} {\?}17 (n-4) \binom{n}{ 9} .\end{equation}
Also at given $g$, for the lowest $n=2g+3$, $S^{(g)}_{2g+3,3} =\frac{4^{g+1}-1}{3}$, OEIS sequence \seqnum{A002450}.
\end{observation}

\begin{observation} The result obtained {in (\ref{simpCHresult})} for $S_{n,k}^{(2)}$ 
 can be generalized in terms of an expression that encodes all (presently) known results for $S_{n,k}^{(g)}$,  $g \ge 1$.  
 This expression is as follows
\resizebox{.9\linewidth}{!}{
  \begin{minipage}{\linewidth}
  \begin{align}
  S_{n,k}^{(g)}\, &\?\; \frac{1}{C(g)} \; \chi^{(g)} *_{4g-4} \Big(\frac{\Gamma(n-t+g-2) \Gamma(n-t+4g-3)}{ \Gamma(k-s-1)\Gamma(k-s+3g-2 )\Gamma(n-k+s-t-2g+1 )\Gamma(n-k+s-t+g)}\Big)
   \\
 &= \frac{1}{C(g)} \; \chi^{(g)} *_{4g-4} \Big(( 3 g-1)!   \binom{ n-k + s - t +g-1}{  n -k+ s -  t -2 g }  
\binom { n - t +g-3}{  k - s - 2}   \binom{ n - t  + 4 g -4}{  n -k + s - t +g - 1 } \Big)
 \label{conjectSnkg}
 \end{align}
  \end{minipage}
}

\noindent with the integer constant $C(g)$ given in (\ref{Coef}) and  coefficients $\chi^{(g)}$ given for $g=3,4$ by 
\scalebox{0.71}{  
$\chi^{(3)}= 
 \begin{matrix}1\cr-8&68\cr28&-340&246\cr-56&612&294&-980\cr70&-340&-3390&4480&245\cr-56&-
   340&5700&-3500&-5530&1464\cr28&612&-3390&-3500&11020&-1824&-1208\cr-8&-340&294&448
   0&-5530&-1824&2944&-16\cr1&68&246&-980&245&1464&-1208&-16&180
   \end{matrix}$,
 $ \chi^{(4)}= 
 \begin{matrix} 
 1 \cr
  -12 & 318 \cr
   66 & -2862 & 6831\cr
    -220 & 11130 & -33651 & 6072 \cr 
 495 & -23850 &  30123 & 156660 & -99693 \cr
  -792 &28620 & \cdots & & &   \cr
 924 & -13356 & &&&\cr   -792 &-13356 &\cr
 495 &28620 \cr
   -220 & -23850 & \cdots \cr
   66 &11130  & 30123 & \cr
    -12 & -2862 &  -33651 &   156660  \cr  
    1 &318 &6831 & 6072 &  -99693 & \cdots
  \end{matrix} $ }\\
where the entries of $\chi^{(3)}$ have been determined from a  subset of the existing data, but some of those of $\chi^{(4)}$ are still undetermined at this stage.
\end{observation}

\begin{observation}
\leavevmode
\begin{itemize} 
\item  Each column of the arrays $\chi^{(g)}$ is symmetric.  
 \item The entries of the first column of $\chi^{(g)}$, which is of length $4g-3$,  are binomial coefficients with alternated signs  {$(-1)^t   \binom {4(g-1)}{ t}$}.  
\item The second column is the product of the line $4(g-1)$  of the triangular array OEIS sequence \seqnum{A144431} (a ``sub-Pascal array'') {by the coefficient $d(g)$ given above in (\ref{dcol2})}. 
\item This second column can also be obtained as $d(g)$ times an appropriate line of a matrix defined as the inverse of the matrix of partial sums of the signed Pascal triangle (see OEIS sequence \seqnum{A059260}). 
 \item The last $(g-1)$ lines of the array $\chi^{(g)}$ have a vanishing sum (a justification is given below). 
 \item The last line of the array $\chi^{(g)}$ is conjectured to be given  by (\ref{chilastline}) (details are given below).
 \end{itemize}
\end{observation}
 
 At genus $g$, the first non-zero coefficients $S_{n,k}^{(g)}$ appear for $(n,k)=(2g+2,2)$, and their ``experimental'' values up to $n=4(2g-1)$, $k=2(2g-1)$ can be used to determine the constants $\chi^{(g)}$ but one can lower these two integers by making use of the previous observations.
 
 As already  mentioned, (\ref{conjectSnkg}), evaluated at $g=1,2$ gives back $S_{n, k}^{(1)}$ and $S_{n, k}^{(2)}$; moreover its evaluation at $g \ge 3$ is compatible with all presently known ``experimental'' results, with the Ansatz (\ref{ansatzk=3}),  and the sum over $k$ of $S_{n, k}^{(3)}$ is indeed equal to  $B_{n}^{(3)}$.
 
 \bigskip
 
This justifies the following conjectures:

\begin{conjecture}[Genus $g=3$ conjecture (weak form)] 
The expression (\ref{conjectSnkg}),  with $g=3$, gives $S_{n, k}^{(3)}$ for all $n$ and $k$.
\end{conjecture}

\begin{conjecture}[Genus $g$ conjecture (strong form)] The expression (\ref{conjectSnkg}),  together with an appropriate triangular array of constants $\chi^{(g)}$ gives $S_{n, k}^{(g)}$ for all $n,k,g>0$.
\end{conjecture}

The corresponding Ansatz on the G.F. is
\begin{equation}  \text{G.F.}\qquad S^{(g)}(x,y) \? \frac{x^{2g+2} y^2\, p^{(g)}(x,y)}{( (1 +x -x y)^2 -4x)^{(6g-1)/2}} \end{equation}
with $p^{(g)}(x,y)=\sum_{\substack{0\le t\le 4(g-1) \\ 0\le s \le t}} \chi^{(g)}(t,s) x^t y^s$ .

We have $S^{(g)}(x,1) =B^{(g)}(x)$, $p^{(g)}(x,1)=P^{(g)}(x)$. The latter polynomial 
being (conjectured) of degree $3(g-1)$ in $x$, this tells us that the last $(g-1)$ lines of the array $\chi^{(g)}$
have a vanishing sum. 

A further conjecture, in accordance with the existing data, 
 is that the terms of highest degree  in  $x$, viz $4(g-1)$, of  $p^{(g)}(x,y)$
  are of the form 
  \begin{equation} [p^{(g)}(x,y)]_{x^{4(g-1)}} = (1-y)^{2(g-1)}  \left[ (1-y)^{4g+1}  y^{-2g-3} \sum_{j=0}^{2g-2} \frac{ 2 \; s_{2g+2+j,\, j+1}}{(2g+j+2)(2g+j+1)}  y^{-j}\right]_+ \label{chilastline}\end{equation}
where $[\cdot]_+$ is the polynomial part in $y$ of the expression and   $s_{p,q}$  are the Stirling numbers of the first kind.
See OEIS sequence \seqnum{A185259} where these polynomials are tabulated.
 If true, this conjecture determines the last line $\chi^{(4)}(13 ,s)$ of $\chi^{(4)}$ to be
 $$\{1 ,318 , 6831 ,6072 ,- 99693 , 103950, 
 196581, - 413820 , 155628 , 146168 ,- 117876 , 
 7776 , 8064 \}. $$

\subsubsection{Particular cases: $S_{n,k}^{(3)}$  and  $S_{n,k}^{(4)}$ for small $k$}
The above general conjecture for $S_{n,k}^{(g)}$ leads, when $g = 3, 4$, and small values of $k=2,3,4$, to simple enough formulae that are displayed below. For $k=2$, they follow from (\ref{Sfork=2}). 
One can check that they are compatible with the known (experimental) values of $S_{n,k}^{(3)}$, up to $n=15$, 
(see tables in the appendix).

\bigskip

{ Genus 3.}
\begin{equation*}
\begin{split}
S_{n,2}^{(3)} &=   \scalebox{0.80} {$\binom{n}{{8}}$} =  \scalebox{0.75} {(0, 0, 0, 0, 0, 0, 0, 1, 9, 45, 165, 495, 1287, 3003, 6435, \ldots)}  \quad \text{See OEIS sequence \seqnum{A000581}.} \\
S_{n,3}^{(3)} &\? \scalebox{0.80} {$17 (n-4) \binom{n}{{9}}$} =  \scalebox{0.75} {(0, 0, 0, 0, 0, 0, 0, 0, 85, 1020, 6545, 29920, 109395, 340340, 935935,\ldots)}\\
S_{n,4}^{(3)} &\?  \scalebox{0.80} {$ \frac{5}{3}{(32 n^2-288 n+613)} \binom{n}{10}$} =  \scalebox{0.75} {(0, 0, 0, 0, 0, 0, 0, 0, 0, 1555, 24145, 194150, 1085370, 4759755, 17482465, \ldots)}
\end{split}
\end{equation*}

{ Genus 4.}
\begin{equation*}
\begin{split}
S_{n,2}^{(4) } &= \scalebox{0.80} {$\binom{n}{10}$}= \scalebox{0.75} {(0, 0, 0, 0, 0, 0, 0, 0, 0, 1, 11, 66, 286, 1001, 3003,\ldots)} \quad  \text{See 
OEIS sequence \seqnum{A001287}.}  \\
S_{n,3}^{(4)}  &\? \scalebox{0.80} {$\frac{341}{6} (n-5) \binom{n}{11}$}=\scalebox{0.75} {(0, 0, 0, 0, 0, 0, 0, 0, 0, 0, 341, 4774, 35464, 186186, 775775,\ldots)}\\
S_{n,4}^{(4)}  &\? \scalebox{0.80} {$\frac{11}{2} (65 \, n \, (n-11) +1842) \binom{n}{12}$}=\scalebox{0.75} { (0, 0, 0, 0, 0, 0, 0, 0, 0, 0, 0, 14421, 252538, 2288286, 14369355, \ldots)}
\end{split}
\end{equation*}

\subsection{Partitions with no singletons: associated numbers $\widehat S_{n, k}^{(g)}$}
\label{partwosingl}

\subsubsection{Genus $0$} See OEIS sequence \seqnum{A108263}.
\begin{equation} \widehat S_{n, k}^{(0)} =  \frac{1}{(n - k + 1)} {\binom{n - k - 1}  { n - 2 k}} \, {\binom{n} {k}}\, .\end{equation}
\begin{equation}{ \text{The ordinary G.F. is } \frac{1+x-\sqrt{(1-x)^2-4 x^2 y}}{2 x (x y+1)} } .\end{equation}

\subsubsection{Genus $1$}
\begin{equation}\widehat S_{n, k}^{(1)} = \frac{1}{6} \, {\binom{k}  {2}} \,  {\binom{n}  {k}}  \, {\binom{n-k}  {k}} . \end{equation}
\begin{equation} \text{The ordinary G.F. is    }     \frac{x^4 y^2}{((1 -  x)^2  - 4 x^2 y)^{5/2}}.\end{equation}

\subsubsection{Genus $2$ and above} 
 It is conjectured that for generic genus $g>0$, the G.F. has the form
\begin{equation}   \widehat S^{(g)}(x,y) {\buildrel ?\over =} \frac{x^{2g+2} y^2 \hat p^{(g)}(x,y)}{((1 -  x)^2  - 4 x^2 y)^{(6g-1)/2}}\end{equation}
with $\hat p^{(g)}(x,y)$ a polynomial  {of degree   $4(g-1)$ in $x$}. 
For instance 
\begin{align}\hat p^{(2)}(x,y) &= 1 + 2 x (-2 + 7 y)  + 
 x^2 (6 - 22 y + 21 y^2) \nonumber \\
 & +x^3 (-4 + 2 y + 7 y^2) + x^4 (1 + 6 y - 19 y^2 + 21 y^3),
\end{align}
as derived in \cite{Z23}.
 
\begin{remark}
 Note that  $p^{(2)}(x,0)=\hat p^{(2)}(x,0)=(1-x)^4$ so that the term of order $y^2$ in $ S^{(2)}(x,y)$ or $ \widehat S^{(2)}(x,y)$,
 \ie the G.F. of genus 2 partitions into two parts with or without singleton, is $x^6 /(1-x)^7=\frac{1}{2} \sum_{n={6}}^\infty \frac{n}{3} x^n \sum_{p=1}^{n-1} \binom{p-1}{ 2} \binom{n-p-1}{ 2}$, in agreement with formula 
 (\ref{conjpn-p}) below. 
 
  By the same token, we may assert  that  $ p^{(g)}(x,0)=\hat p^{(g)}(x,0)=(1-x)^{4(g-1)}$ so that the term of order $y^2$ in $ S^{(g)}(x,y)$ or $ \widehat S^{(g)}(x,y)$ is 
  \[ x^{2g+2} /(1-x)^{2g+3}=\frac{1}{2} \sum_{n=2g}^\infty \frac{n}{g+1}  x^n \sum_{p=1}^{n-1} \binom{p-1}{ g} \binom{n-p-1}{ g}.\]
   This implies that 
   \begin{equation}\widehat S_{n, 2}^{(0)} =  S_{n, 2}^{(0)}  -n = \frac{1}{2} n (n-3) \quad \text{and}  \quad   \widehat S_{n, 2}^{(g)} = S_{n, 2}^{(g)} = \binom{n}{ 2g+2} \; \text{for}\; g >0\end{equation}
 in agreement with the result (\ref{Sfork=2}).
  These numbers can be recognized as the elements of the array in OEIS \seqnum{A275514}.
  \end{remark}

\subsubsection{Particular case $n=2k$} \label{sec:HZ:wideStirling}
Since we assume in this section that there are no singletons and since $k$ denotes the number of parts, the equality $n=2k$ implies that each part is of length $2$, so that the type is determined,  $[\alpha] = [2^k]$, and {$\widehat S_{2k,k}^{(g)}  = C_{2k, [2^k]}^{(g)}$}.
Because of this coincidence, we postpone the study of this particular case to the next section, which is devoted to the study of coefficients $C_{n, [\alpha]}^{(g)}$.

\section{Genus-dependent Fa\`a di Bruno coefficients $C_{n, [\alpha]}^{(g)}$. \\  Part I. Fully solved cases}
\label{sect5}

 The genus-dependent Fa\`a di Bruno coefficients $C_{n, [\alpha]}^{(g)}$ are  explicitly known in many specific cases, for particular types  $[\alpha] = [\cdots \ell^{\alpha_\ell} \cdots]$ and/or for particular values of the genus $g$, most of them discussed and summarized in Section \ref{sect6}.
However, to  the best of our knowledge, they are generically known in only three families of cases, two of them are classic---the cases of genus 0, for  all types, 
and the partitions of type $[2^p]$, for all  $g$---and the third one is new, the partitions into two parts, \ie of type $[p,n-p]$, for all $g$.
We review these three cases in this section. 
In addition, the G.F. of all types of partitions are explicitly known for genus 1 and 2  (see \cite{Z23}), 
although the extraction of explicit formulae for the Fa\`a di Bruno coefficients is arduous.

\subsection{The particular case $g=0$}

$C_{n,[\alpha]}^{(0)}$ is the number of {\it noncrossing} partitions (also called planar partitions) of type $[\alpha]$.
\begin{equation} \label{genus0}C_{n,[\alpha]}^{(0)}=\frac{n!}{(n+1-\sum \alpha_\ell)!  \ \prod_\ell \alpha_\ell!}=\inv{n+1}  \binom{n+1}{ \alpha_1,\ldots ,\alpha_{n}, n+1-\sum \alpha_j}\
 \end{equation}
 where the symbol $({}_{\cdots}^{\cdots})$ denotes a multinomial coefficient. 
 It was first derived by Kreweras \cite{Krew}, and reappeared later in the context of large random matrices
 \cite{BIPZ} and of free probabilities and their free (or noncrossing) cumulants \cite{Speicher}.
 One may also collect these expressions into a G.F. 
\begin{equation}\label{defZ0} Z^{(0)}(x)=1+\sum_{n\ge1} x^n\sum_{[\alpha]\vdash n}C_{n,[\alpha]}^{(0)}  \prod_{\ell =1}^n \kappa_\ell^{\alpha_\ell}\end{equation}
 where the $\kappa_\ell$ are new indeterminates, from which we may also 
 construct the function 
 \begin{equation}\label{defW} W(x)=\sum_{\ell \ge1} \kappa_\ell x^\ell. \end{equation} 
 Then, it was shown in \cite{BIPZ}  that (\ref{genus0}) 
 is equivalent to the following functional relation between $Z^{(0)}$ and $W$
 \begin{equation}\label{eqZ0gen}  Z^{(0)}(x)=1+ W(x\, Z^{(0)}(x)). \end{equation} 
 Also see \cite{Cvitanovic} for a nice graphical interpretation of that identity.

 \subsection{Genus 1 and 2}As recalled above, the genus 1 and 2 G.F. have been constructed in \cite{Z23}. 
  We shall use them in the following to substantiate some remarks and conjectures. 
For illustration, we recall here the expression of the G.F. in genus 1. Let
 \begin{equation} V(x)=xW'(x)\, ,\quad  X_2(x)= x W'(x)-W(x)\, ,\quad Y_2(x)= \frac{1}{2} x^2 W''(x)\, ,\quad \tilde x=x Z^{(0)}(x).\end{equation}
 Then
 \begin{equation}\label{genus1} Z^{(1)}(x) = \frac{X_2(\tilde x) Y_2(\tilde x)}{(1-X_2(\tilde x))^4(1-V(x))}. \end{equation}
 The more cumbersome expression of  $Z^{(2)}(x)$ will not be recalled here.
 
 Note that even though we have explicit expressions of their G.F., the determination of Fa\`a di Bruno 
 coefficients for $g=1$ or 2 is still implicit, contrary to formulae (\ref{genus0}).

\subsection{Type $[\alpha]=[2^k]$}  So $n=2k$ ($k$ parts of length $2$) and $g \le \frac{k}{2}$.
\label{sec:HZ}

If we focus on the terms with $[\alpha]=[2^k]$, it suffices to specialize the indeterminates $\kappa$ to
$\kappa_\ell=\kappa_2 \delta_{\ell,2}$. By a small abuse of notation, we still use $Z^{(g)}(x)$ and $W(x)$ for these
specialized G.F. 
As already mentioned { and explained at the end of Section \ref{partwosingl}}, $C_{2k, [ 2^k]}^{(g)}$ is known for all $g$ and coincides with $\widehat S_{2k,k}^{(g)}$. 
This famous case was first solved by Walsh and Lehman \cite{WL1, WL2} by combinatorial methods;  then 
 by Harer and Zagier \cite{HZ}, in the context of the virtual Euler characteristics of the moduli space of curves, by means of matrix integrals;  and by Jackson by a character theoretic approach \cite{Jackson}. It has been the object of
 an abundant literature since then, {a good review of which is given in \cite{LZ}. Also see \cite{Chapuy-th}. The reason for which  this case can be solved for arbitrary genus is that the 
 crucial constraint of monotonicity of the cycles of $\tau$ is here irrelevant, and we are just dealing with ordinary maps. One finds that}
\begin{equation}
\widehat S_{2k,k}^{(g)}  = C_{2k, [ 2^k]}^{(g)}= \frac{(2k)!}{(k+1)! (k-2g)!} \left[\left(\frac{u/2}{\tanh u/2}\right)^{k+1}\right]_{u^{2g}}
\label{HZeq}
\end{equation}
 where the notation $[Y]_{u^k}$ means the coefficient of $u^k$ in expression $Y$.
The first few terms are given in Table \ref{table:S2kkgValues}.
 
 \begin{table}[ht]
 \begin{center}
{\small
\begin{tabular}{l | c| c|  c| c|  c| c|  c| c|  c| c| }
$\ \hfill\hfill g$ & 0 & 1 & 2 & 3 &4 \\ 
\hline
$k=1$ & 1 &&&& \\
$k=2$ & 2 & 1&&& \\
$k=3$ & 5 & 10 &&& \\
$k=4$ &14 & 70 & 21&& \\
$k=5$ &42 & 420 & 483&& \\
$k=6$ &132 & 2310 & 6468 & 1485& \\
$k=7$ &429 & 12012 & 66066 & 56628& \\
$k=8$ & 1430 &60060 & 570570 & 1169740 & 225225
\end{tabular}}
\end{center}
\caption{Values of  $\widehat S_{2k,k}^{(g)}$.}
\label{table:S2kkgValues}
\end{table}

The $g=0$ column is, by (\ref{genus0}): $C_{2k, [2^{k}]}^{(0)} =\inv{k+1} \binom{2k}{ k} = {\mathcal C}_k$ (Catalan numbers)\footnote{It therefore  coincides with~$B_k^{(0)}$.}, whose G.F. is 
\begin{equation} Z^{(0)}(u)=1+  \sum_{k=1}^\infty C_{2k, [2^{k}]}^{(0)}    u^k= \frac{1-\sqrt{1-4u}}{2u}\end{equation}
which satisfies 
\begin{equation}\label{eqZ0} Z^{(0}(u) =1 + u (Z^{(0}(u) )^2.\end{equation}
(This is the equation (\ref{eqZ0gen}) expressed here for $W(x)=\kappa_2 x^2$ in the variable $u=\kappa_2 x^2$.)
The $g=1$ column  is $\frac{(2k-1)!}{6 (k-2)!(k-1)!}= \binom{2k-1}{ 3} \, {\mathcal C}_{k-2} = \frac{(k+1)k(k-1)}{12} \, {\mathcal C}_k$.  See OEIS sequence \seqnum{A002802}.
The $k$-th row's sum is, by (\ref{faadiBruno}), given by  $(2k-1)!!$, viz $\{1, 1,3,15, 105,945,\ldots\}$. 
More generally,  
\begin{equation}\label{polR} C_{2k, [2^{k}]}^{(g)} =\inv{2^g }\, {\mathcal C}_k \, R_g(k)\end{equation} with $R_g(k)$ a polynomial of degree $3g$ in $k$~\cite{WL1, HZ} which 
(for $g>0$) vanishes for $k=-1,0,\ldots, 2g-1$, and whose form can be made explicit \cite{GoupilSch, Chapuy-th}.
There are several expressions for the values that it takes when its argument $k$ is an arbitrary non-negative integer. 
One of them, in terms of  {\sl unsigned} Stirling numbers of the {\sl first} kind  $c_{p,q}$\footnote{Here $c_{p,q}$ is the number of permutations of $p$ elements that have $q$ distinct cycles.  
They are positive integers such that $s_{p,q}=(-1)^{p-q} \, c_{p,q}$ where the $s_{p,q}$, the Stirling numbers of the first kind, obey 
$\sum_{p\ge 1} S_{n,p}\, s_{p,q} = \delta_{n,q}$.}, are as follows  \cite{ChenReidys}:
\begin{equation} R_g(k) = \sum_{s=0}^k  {\binom{k}  {s}} \sum_{j=0}^{k+2-2g}\, (-1)^{s+1-j}  \;  c_{k - s + 1, k + 2 - 2 g - j} \;  c_{s+1,j}.\end{equation}

As noticed in \cite{HZ}, the numbers $C_{2k, [ 2^k]}^{(g)}$ (called $\epsilon_g(k)$ there) satisfy a recurrence formula
\begin{equation} (k+1) \label{recur-HZ}  C_{2k, [ 2^k]}^{(g)} = 2(2k-1)C_{2(k-1), [ 2^{k-1}]}^{(g)}+\inv{2} (2k-1)(2k-2)(2k-3) C_{2(k-2), [ 2^{k-2}]}^{(g-1)}\end{equation}
or in terms of the polynomials $R$ introduced in (\ref{polR}),
\begin{equation} R_g(k)=R_g(k-1) +\binom{k}{ 2} R_{g-1}(k-2).\end{equation}
{In \cite{Chapuy-th}, Chapuy gave a combinatorial interpretation of that formula, from which 
he derived another recurrence equation
\begin{equation} 2g C_{2k, [ 2^k]}^{(g)} =\sum_{h=1}^{g}  \binom {k+2h+1-2g}{2h+1} C_{2k, [ 2^k]}^{(g-h)}.\end{equation}} 

\begin{proposition}
 The generating function of the $C_{2k, [2^k]}^{(g)}$ for $g>0$ is of the form
\begin{equation}\label{GFHZ}{Z}^{(g)}(u) :=\sum_{k} C_{2k, [2^k]}^{(g)} u^k =   \frac{{u^{2g}}\, Q^{(g)}(u)}{(1-4 u)^{(6g-1)/2}},\end{equation}
where $Q^{(g)}(u)$ is a polynomial of degree $g-1$ in $u$ satisfying 
\begin{equation} \label{value0} Q^{(g)}(0)=  \frac{(4 g)!}{2^{2 g}(2 g+1)!}.\end{equation}
\label{oldprop1}
\end{proposition}

\begin{proof} 
One finds by explicit calculation that $Q^{(1)}(u)=1$. Equation (\ref{recur-HZ})  implies that $Q^{(g)}(u)$ satisfies the following 
recurrence formula
\begin{align} \nonumber &  \!\!\!\!\!\!{\scriptstyle (1-4 u) u \frac{d}{du}Q^{(g)}(u)+\big(   2g+1 +4 u(g-1)\big) Q^{(g)}(u) }=\\
\nonumber &
{\scriptstyle  \big(3 \binom{4g-1}{ 3} +6 (-1 + 4 g) (1 - 14 g + 8 g^2)u +96 (g-2 ) ( 4 g^2 - 8 g -1) u^2
+128(g-2)(g-3)(2g-5)u^3 \big)Q^{(g-1)}(u)}  \\
  \nonumber  &
  {\scriptstyle + \big( 48 g^2  - 24 g +3+ 24( 8 g^2 - 20 g -1  ) u + 
192 (g-3)^2   u^2\big)   u(1-4u) \frac{d}{du}Q^{(g-1)}(u)} \\
  \label{recureq} &
   {\scriptstyle   +24 \big(g +u(2g - 7) \big) u^2(1-4u)^2 \frac{d^2}{du^2}Q^{(g-1)}(u) + 4u^3(1-4u)^3 \frac{d^3}{du^3}Q^{(g-1)}(u) }
\end{align}
that we abbreviate by
$$ {\mathcal D} Q^{(g)}(u) = \widehat{\mathcal D} Q^{(g-1)}(u) $$
with two linear differential operators in the variable $u$, ${\mathcal D}$   and $ \widehat{\mathcal D}$.
Equation (\ref{recureq}) for $u=0$ fixes the ratio   $Q^{(g)}(0)/Q^{(g-1)}(0) =\frac{(4 g-3) (4 g-2) (4 g-1)}{2 (2 g+1)}$, which leads to the expression 
(\ref{value0}) above. 
  The proof of (\ref{GFHZ}) is obtained by induction assuming that $Q^{(g-1)}(u)$ is a polynomial of degree $g-2$. 
The polynomiality property  of  $Q^{(g)}(u)$ is then proved as follows: the rhs of (\ref{recureq}) might seem to be of degree $g+1$, but in fact its degree is only $g-1$ because of trivial identities
\begin{align}\nonumber
&{\scriptstyle \left[ \widehat{\mathcal D}  u^{g-2}\right]_{u^{g+1}}=128(g-2)(g-3)(2g-5)-4 \times 192 (g-2)(g-3)^2 +4^2\times 24 (2g-7) (g-2)(g-3) +4\times(-4)^3 (g-2)(g-3)(g-4)}
\\ \nonumber
&\qquad\qquad\  = {\scriptstyle128 (g-2)(g-3) \Big(   (2g-5) - 6  (g-3)   +3 (2g-7) -2 (g-4)\Big) \equiv 0}\\[4pt]
&{\scriptstyle  \left[ \widehat{\mathcal D}  u^{g-3}\right]_{u^{g}}= 128(g-2)(g-3)(2g-5)-4 \times 192 (g-3)^3 +4^2\times 24 (2g-7) (g-3)(g-4) +4\times(-4)^3 (g-3)(g-4)(g-5) \equiv 0}\\  \nonumber 
& {\scriptstyle  \left[ \widehat{\mathcal D}  u^{g-2}\right]_{u^{g}}=96 (g-2) \left(4 g^2-8 g-1\right) +(g-2) \left(192 (g-3)^2-96 \left(8 g^2-20 g-1\right)\right)+1344 (g-3) (g-2) +192 (g-4) (g-3) (g-2) \equiv 0},\end{align}
using the notation introduced in (\ref{HZeq}). 
Likewise, in the lhs,  $\left[ {\mathcal D}  u^{g-1}\right]_{u^{g}}\equiv 0$, so that it is consistent to look for a polynomial solution of degree $g-1$ for $Q^{(g)}(u)$.
Write $ Q^{(g)}(u) =\sum_{r=0}^{g-1}  q_r u^r$. Then equation (\ref{recureq}) 
recursively determines the coefficients $q_r$ in terms of those of $Q^{(g-1)}(u)$
$$ \left[{\mathcal D} Q^{(g)}(u)\right]_{u^r}  =  (2g+1+r) q_r + 4 (g-r) q_{r-1} = \left[{\mathcal D} Q^{(g-1)}(u)\right]_{u^r} \quad \mathrm{for}\ 1\le r \le g-1.$$
Together with the value of $q_0= Q^{(g)}(0)$ given above, this completely determines all  coefficients $q_r$ and completes the proof that 
$Q^{(g)}(u)$ is a polynomial.\end{proof}

\medskip
Note that the expression $(1-4u)$ in the denominator of (\ref{GFHZ})  
is---once again---nothing else than the discriminant of the equation (\ref{eqZ0}) satisfied by $Z^{(0)}$. 

The first few $Q^{(g)}$ are
\begin{align} \label{GFHZpol} \nonumber Q^{(1)}(u)=1\ ;\quad Q^{(2)}(u)=21(1+u)\ ;\quad Q^{(3)}(u)= 11(135 +558 u + 158u^2)\ ;\\
 Q^{(4)}(u)=11\times 13 (1575 + 13689 u +18378 u^2 +2339 u^3)\ ;\qquad\qquad\\  \nonumber
 Q^{(5)}(u)=  3 \times 13 \times 17 \times 19 \, (4725 + 67620 u+ 201348 u^2 +132356 u^3 +9478 u^4)\ ;\, \cdots  \end{align}
  $\text{with} \quad Q^{(g)}(0)=\{1, 21, 1485,225225,\ldots\} = C_{4g, [2^{2g}]}^{(g)} = \inv{2^g }\, {\mathcal C}_{2g} \, R_g(2g) =  \frac{(4 g)!}{2^{2 g}(2 g+1)!}.$\\
 See OEIS sequence \seqnum{A035319}.

\bigskip

\subsection{Partitions into two parts: type $[\alpha]={[p, n-p]}$} 
\label{2part-partitions}

In this section, we make use of both the description of   partitions by pairs $(\sigma, \tau)$ of permutations and the diagrammatic representation presented in Section \ref{sectdefgenus}  and apply them
 to the case of a partition of $\{1,\ldots,n\}$ into two parts of size $p$ and $n-p$. Thus in the diagrammatic representation, 
 $n$ points lie on the circle, numbered and clockwise ordered 
 from 1 to $n$, while two vertices of valence $p$ and $n-p$, associated by the two increasing cycles of $\tau$ are inside the disk, 
 and their edges connect to the points on the circle, 
 {\it with no crossing of edges originating from the same vertex}.  The faces of the map correspond to the cycles of the 
 permutation $\sigma\circ \tau^{-1}$. 

\begin{example}
 Figure \ref{Fig1Comp}(a) shows a diagrammatic representation of a partition  of
 $\scalebox{0.95}{\{1,\ldots,12\}}$   into two parts $\alpha=\scalebox{0.95}{(\{1,4,5,8,10\},\{2,3,6,7,9,11,12\})}$. 
  Thus  $\tau=\scalebox{0.95}{(1,4,5,8,10),(2,3,6,7,9,11,12)}$,  $\sigma=\scalebox{0.95}{(1,2,\ldots,12)}$,
and  the product $\sigma \circ \tau^{-1} = \scalebox{0.95}{(1,11,10,9,8,6,4,2)(3)(5)(7)(12)}$ has 5 cycles.
 The map has $f=5$ faces, which can also be seen by following the circuit indicated by arrows on the figure,
  and its genus is thus~$3$  by applying (\ref{genus}).
\begin{figure}[!ht]
\begin{center}
\includegraphics[width=.6\textwidth]{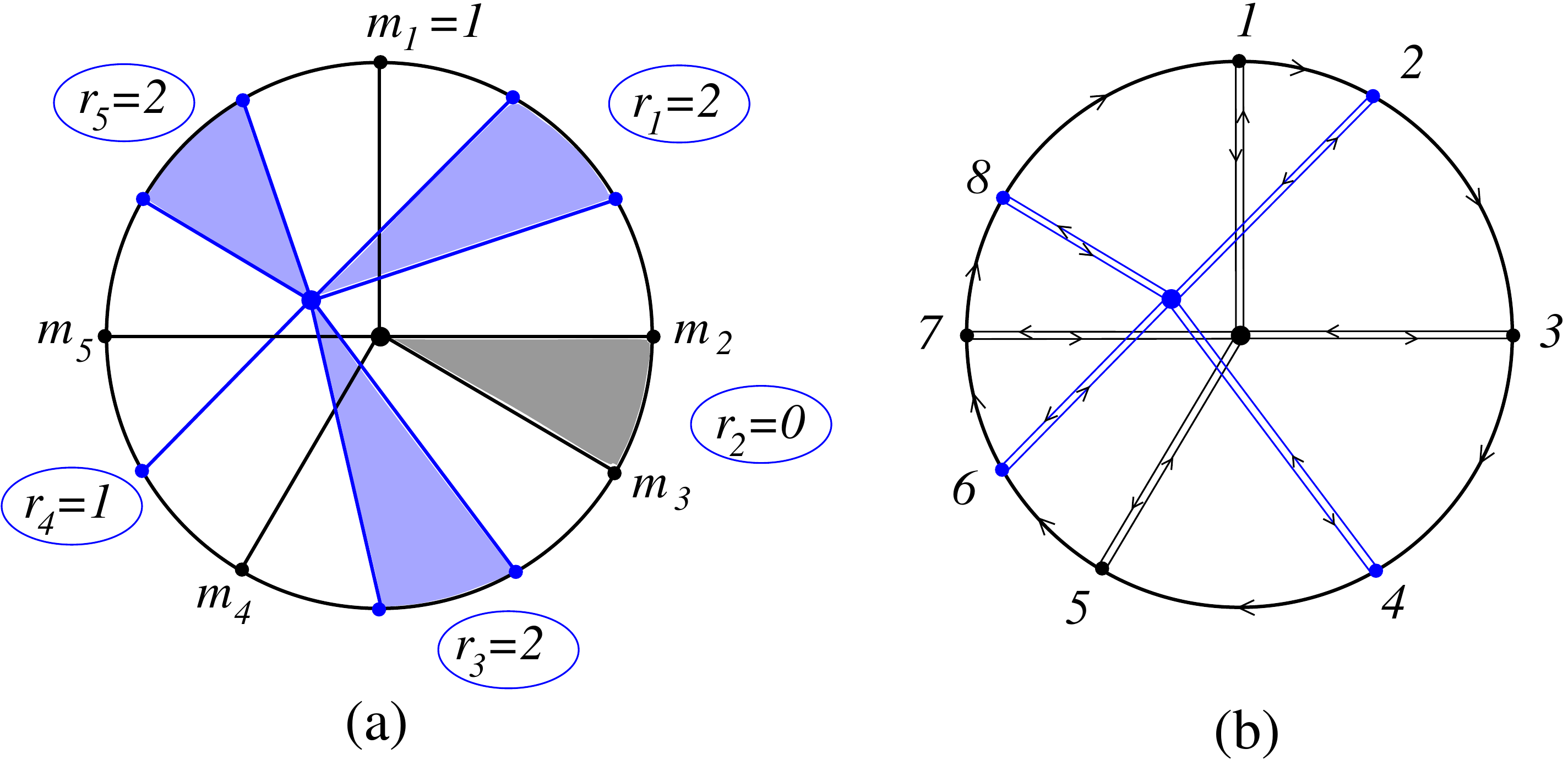}
\end{center}
\caption{(a) Partition with $n=12$, $p=5$, $f=5$, $s_1=1$, $s_2=3$, $g=3$;  
(b)  removing shaded faces, \ie singletons of $\sigma\circ \tau^{-1}$,  relabeling the points
and  doubling the edges to make the one cycle (or face) $(1,8,7,6,5,4,3,2)$  more visible,  $f'=1$.}
\label{Fig1Comp}
\end{figure}
\end{example}
  
Let  $C_1$ and $C_2$ denote the two (increasing) cycles of $\tau$,   $m_1 < m_2 <\cdots < m_p$  the $p$ elements of  $C_1$, 
and  $p_1 <p_2<\cdots<p_{n-p}$ those of $C_2$. 
(On the figure, we chose $m_1=1$ and the corresponding points on the circle are black, those of $C_2$ are blue.)  

Let us consider a singleton $x$ of $\sigma\circ\tau^{-1}$:  $x = \sigma(\tau^{-1}(x)) =\tau^{-1}(x)+1$, \ie $x=\tau(x-1)$, meaning that $x-1$ and $x$ 
belong to the same cycle of $\tau$. Let $s_1$, resp., $s_2$, be the number of such singletons in cycle $C_1$, resp., $C_2$. Diagrammatically, 
they count the numbers of pairs of edges exiting the same vertex and reaching  consecutive points on the circle,
or the numbers of grey, resp., blue shaded faces on the figure.   
Let $r_i$ be the number of  integers strictly between $m_i$ and $m_{i+1}$, $i=1,2,\ldots,p$, 
(with the convention $m_{p+1}\equiv m_1$). Thus $\sum_i r_i= n-p$.  
For each of the $s_1$ indices $i$ for which $r_i=0$, (\ie $m_{i+1}=m_i+1$), 
there is a singleton of  $\sigma\circ\tau^{-1}$ in $C_1$  (and the associated face on the figure is grey shaded); 
 for each $i$ for which  $r_i\ge 1$, there are $r_i-1$ singletons in 
$C_2$ (\ie blue faces attached to the second vertex). 
In our example of  Fig.~\ref{Fig1Comp} there are respectively one and three such faces. 
The total number of singletons (= shaded faces) is thus
$$  \sum_{\substack{i \\ r_i=0}}1 + \sum_{\substack{i\\ r_i\ge 1}}(r_i-1) = s_1 +  \big((n-p) -(p-s_1)\big) = n-2p+2 s_1.$$
Interchanging the roles of $C_1$ and $C_2$ tells us that it is also equal to $n-2(n-p) +2s_2$, hence that
$$ s_2= s_1+ n-2p.$$ 
In addition to these singletons, $\sigma\circ\tau^{-1}$ has $f'$ other cycles (= unshaded faces), and a total of 
$ f= s_1 + s_2 +f'= n-2p+2s_1+f'$ faces, whence the genus 
\begin{equation} g= \frac{n-1-f}{2}= p- s_1-\frac{f'+1}{2} \le p-1. \qquad \qquad\end{equation}
or by symmetry between the two vertices, $g\le n-p-1$, which
shows that $g\le \min(p-1, n-p-1)$.
Clearly the minimum value $g=0$ is reached for $f'=1,\ s_1=p-1$, \ie with the $p$-vertex being attached to $p$ consecutive points on the circle.
In fact  the  previous inequality $f'\ge 1$  is  an equality.  In other words, we have, with the previous notation, the following lemma:
\begin{lemma} 
\label{lemmalabelfprime}
For any partition of type $[p,n-p]$, we have $f'=1$,  $g=p-s_1-1= (n-p) -s_2-1$.\end{lemma}  
\begin{proof} In the spirit of the reduction of diagrams  to irreducible ones of the same genus, in Cori--Hetyei procedure~\cite{CoriH13, CoriH17, Z23}, we now 
define a new partition of a new set: the $s_1$  points $m_{i+1}=m_i+1$ and $s_2$ points   $p_{j+1}=p_j+1$ are deleted, $n$ reduced to
$n'=n-s_1-s_2$, $p$ to $p'=p-s_1 =n-p-s_2$, all entries relabeled from 1 to $n'$, but the genus is unchanged $2g'=n'-f'-1=2g$. We are now dealing with a partition of $\{1,\ldots, n'\}$ in 
two parts of equal length $p'$, with associated permutations $\sigma'$, cyclic on $\{1,\ldots, n'\}$, and $\tau'$ of type $[p',p']$ with two increasing cycles
 $C'_1$ and $C'_2$.
By construction, $\sigma'\circ \tau'^{-1}$ has no singleton. We claim that such a  $\sigma'\circ \tau'^{-1}$ has only one cycle, hence is a cyclic permutation.

Since we assume that $\sigma'\circ \tau'^{-1}$ 
has  no singleton, this means that for any $i_1\in C'_1$, say,  $i_2:=\sigma'(\tau'^{-1}(i_1))= \tau'^{-1}(i_1)+1 $ satisfies 
$\tau'^{-1}(i_1) < i_2 < i_1$ since each cycle of $\tau'$ is increasing and no singleton appears. 
 But again by the monotonicity of each of the cycles of $\tau'$,  there cannot be between $\tau'^{-1}(i_1)$   and $ i_1$
 an element in the same cycle $C'_1$. This means that $i_2$ belongs to the other cycle $C'_2$.
 
 Suppose that $i_2 <i_1-1$ and consider the point $i_2+1=\sigma'(i_2)$.
   Being between  $\tau^{-1}(i_1)$  and $ i_1$  it cannot belong to the same cycle  $C'_1$ as $i_1$.
 It cannot either belong to cycle $C'_2$ because it would then equal $\tau'(i_2)$ and thus make a singleton of  $\sigma'\circ \tau'^{-1}$:
 $\sigma'(\tau'^{-1}(i_2+1))= \sigma'(i_2)=i_2+1$. Thus the hypothesis $i_2 <i_1-1$
 must be rejected and we conclude that $i_2= \sigma'(\tau'^{-1}(i_1)) =i_1-1=\sigma'^{-1}(i_1)$. Thus $\sigma'^2(\tau'^{-1}(i_1))=i_1$ and
 since this applies to arbitrary $i_1$, we conclude that 
 $$ \sigma'^2= \tau'$$
 \ie that $\tau'$ is the product of two cyclic permutations acting on even, resp., odd integers between 1 and $n'$; 
 or, in other words, that   $\sigma'\circ \tau'^{-1}$ is the cyclic permutation  $\sigma'^{-1}$.  
\end{proof}

 The property $f^\prime =1$, established in the above lemma, may be rephrased as follows:\\
 {\sl Let $\tau$ be a permutation of ${\mathcal S}_n$ defined as the product of two disjoint increasing cycles, 
 and $\sigma$ be the circular permutation $(1, 2, . . . , n)$, then the cyclic decomposition of the product $\sigma\circ\tau^{-1}$ contains 
 only one non-trivial cycle.}
 
We could not find this property mentioned in the literature.
 
\begin{proposition}{\label{propostwopartslabel}For general $n\ne 2p,  g$, and $p\ge 2$, (otherwise, if $n=2p$, multiply by $\inv{2}$)
\begin{equation} \label{conjpn-p}C_{n, [p,n-p]}^{(g)}=\frac{n}{g+1} \binom{p-1}{ g}  \binom{n-p-1}{ g}\, .\end{equation}}\end{proposition}
Note that this result is symmetric under $p\leftrightarrow n-p$ as it should, and that the bound $g\le \min(p-1,n-p-1)$ is manifest. 
An alternative expression is
\begin{equation} \label{conjpn-p-alt}C_{n, [p,n-p]}^{(g)}=\frac{n}{p} \binom {p}{ p-1-g}  \binom{n-p-1}{ g}.\end{equation}
\begin{proof} 
  Following the lines of Lemma \ref{lemmalabelfprime}, a {partition} is completely determined by the choice of the $p$ points $m_i$ on the circle, subject to the condition that there are $s_1$ pairs of adjacent points. Thus
 \begin{equation} C_{n, [p,n-p]}^{(g)}=\#\{ m_1,\ldots, m_p \in \{1,n\}\,|\, \#\{i|m_{i+1}=m_i+1  \}=s_1 \} ,\end{equation}
where $g=p-s_1-1$.
 This number may be easily computed by a transfer matrix technique\footnote{We  are quite grateful to Philippe Di Francesco for this suggestion.}. Suppose the black and blue points of the circle 
are representing two states of a periodic system on a circle, and assign a weight $1, t, t^2u$ to a
transition (\ie an arc on the circle) between  respectively black-black, black-blue and blue-blue points.
The matrix $$ M=\begin{pmatrix}1& t\\ t & t^2 u\end{pmatrix}$$ describes the possible transitions between these states.
 The number $C_{n, [p,n-p]}^{(g)}$ is then the coefficient of $t^{2p}  u^{p-1-g}$ in 
\begin{equation} t_n={\rm tr}\, M^n. \end{equation} 
 Let $z:= t^2$. By virtue of the characteristic equation of $M$, the numbers $t_n$ satisfy the recurrence relations
 \begin{equation}\label{recurtn} t_n = t_1 t_{n-1} +z(1-u) t_{n-2}\end{equation}
 with $t_0=2 $ and $  t_1=1 + z u$, {whence $t_2=1+2 z +z^2 u^2$,
 $t_3=1 + 3 z + 3 z^2 u + z^3 u^3$}. It follows from (\ref{recurtn}) that 
\begin{equation}  [t_n]_{z^p u^{p-1-g}}= [t_{n-1}]_{z^p u^{p-1-g}} + [t_{n-1}]_{z^{p-1} u^{p-2-g}}+[t_{n-2}]_{z^{p-1} u^{p-1-g}} -  [t_{n-2}]_{z^{p-1} u^{p-2-g}},  \end{equation}
 from which all $ [t_n]_{z^p u^{p-1-g}}$ may be reconstructed for $p\le n-1, 0\le g\le p-1$. 
 
 Let $D_{n,p,g}:= \frac{n}{g+1} \binom{p-1}{ g}  \binom{n-p-1}{ g}$, one may check that the $D$'s  
satisfy the same relation, namely 
\begin{equation}   D_{n,p,g}= D_{n-1,p,g}+ D_{n-1,p-1,g}+ D_{n-2,p-1,g-1} - D_{n-2,p-1,g},\end{equation}
with
\scalebox{0.84} {$D_{2,1,0}=[t_2]_{z u^0}=2,\  D_{2,1,1}=[t_2]_{z u^{-1}}=0,\ D_{3,2,0}=[t_3]_{z^2 u}=3,\ D_{3,2,1}=[t_3]_{z^2 u^0}=0,\ D_{2,1,0}=[t_2]_{z u^0}=0\,$.}
Thus $ C_{n, [p,n-p]}^{(g)}$ is equal to $ D_{n,p,g}$, 
which completes the proof of Proposition \ref{propostwopartslabel}. \end{proof} 

\begin{remark}
The formula (\ref{conjpn-p}), originally proposed for $p\ge 2$, extends trivially to all $p$: from (\ref{with/osingle}), we
 learn that $C^{(g)}_{n,[1,n-1]}= n  C_{n-1,[n-1]}^{(g)}= n \delta_{g\, 0}$, in accordance with the rhs of  (\ref{conjpn-p})
 evaluated at $g=0$. \end{remark} 
\begin{remark}
One may check that the expression of $C_{n,[p,n-p]}^{(g)}$ 
 is consistent with the Fa\`a di Bruno coefficient (\ref{faadiBruno}):  
$\sum_{g=0}^{p-1}  C_{n,[p,n-p]}^{(g)} = \binom{n}{ p}$. Once again, this is trivially true for $p=1$ and, 
for $2\le p < n-p$, this is an easy consequence of the celebrated Vandermonde identity, namely:
\begin{equation}\label{VdMidentity}\sum_{k=0}^m \binom{n}{ k} \binom{\ell}{ m-k} = \binom{\ell+n}{ m}. \end{equation} \vskip4pt
Here, using  (\ref{conjpn-p-alt})  together with (\ref{VdMidentity}), one finds:
\begin{equation}  \frac{n}{p}\sum_{g=0}^{p-1}\binom{p}{ p-1-g} \binom{n-p-1}{ g}=
\frac{n}{p} \binom{n-1}{ p-1} =\binom{n}{ p}. \end{equation}
\end{remark}
\begin{remark}
Generating function of the  $C_{[p,n-p]}^{(g)}$. One may build a G.F. for the $ C_{[p,n-p]}^{(g)}$ 
adapted to their symmetry under $p\leftrightarrow n-p$: 
\begin{equation}  {Z}^{(g)}(x,v):= \frac{1}{2} \sum_{n=0}^\infty \sum_{p=1}^{n-1} C_{[p,n-p]}^{(g)}  x^n v^{2p-n}\end{equation}
One finds 
\begin{equation} {Z}^{(g)}(x,v) = \mathcal{Z}^{(g)}(x,1/v) =\frac{(2-x(v+1/v) )x^{2g+2}}{2(1-x v )^{g+2}(1-x/v)^{g+2}}.\end{equation}
In particular, for $v=1$, one recovers the G.F. of the $S^{(g)}_{n,2}$ already encountered in the 
Remark at the end of Section~\ref{partwosingl}, $\sum_n S^{(g)}_{n,2} x^n=\frac{x^{2g+2}}{(1-x)^{2g+3}}$.
 \end{remark} 
 \begin{remark}
  Proposition \ref{propostwopartslabel} can be checked at low $n$,  for instance $n=5$,  by determining explicitly the partitions contributing 
 to $C_{n, [n-5,5]}^{(g)}$. Their number indeed starts as in Table \ref{tabletwoparts5}, in agreement with Eqs.~\eqref{conjpn-p}--\eqref{conjpn-p-alt}.
 \end{remark} 

\begin{table}[h]
\centering
\begin{tabular}{l | c| c|  c| c|  c| c|  c| c|   }
\qquad $g$& 0 & 1 &2 &3& 4\\ 
\hline
$n=6$ & 6&&&&\\
$n=7$ & 7&14&&&\\
$n=8$ & 8&32&16&&\\
$n=9$ & 9&54&54&9&\\
$n=10$ & 5&40&60&20&1\\
$n=11$ &11&110&220&110&11\\
$n=12$ &12&144&360&240&36\\
$n=13$ &13&182&546&  455&91\\
$n=14$ &  14&224&784&784&196\\
$n=15$ & 15&270&1080&1260&378\\
\end{tabular}
\caption{Number of partitions contributing to $C_{n, [n-5,5]}^{(g)}$.}
\label{tabletwoparts5}
\end{table}

\subsubsection{Case $[\alpha] = [p^2]$}
So $n=2p$ ($2$ parts of length $p$), and $g=\frac{2p-1-f}{2} \le p-1$.
This is of course a particular case of the type  $[\alpha] = [p,n-p]$ considered in this section.
 Here, the Fa\`a di Bruno coefficients are the numbers ``of ways to put $p$ identical objects into $g+1$ of altogether $p$ distinguishable boxes''.
See OEIS sequence \seqnum{A103371}.  Notice that the last writing below is indeed consistent with (\ref{conjpn-p}). One has 
\begin{equation} C_{2p, [p^2]}^{(g)}= \binom{p-1}{ p-g-1} \binom{p}{ p-g-1}=\binom{p-1}{ g} \binom{p}{ g+1}=\frac{p}{g+1} \binom{p-1}{ g}^2\end{equation}
For small values of $p$ and $g$ these coefficients are gathered with those of the cases $[\alpha]=[p^k]$ studied in Section  \ref{tablepwedgek}.

 The general formula being given above we only notice that the  $g=0$ sequence is just $p$ and that the $g=1$ sequence defines the pentagonal pyramidal numbers that we shall meet again in Section  \ref{tablepwedgek} 
---see our comments there. Notice also that the sum over $g$   
is $\binom{2p-1}{ p}=\{1,3,10,35,126, 462, \ldots\}$,(see {OEIS sequence} \seqnum{A001700}), and that
the penultimate term in each row, $\{2,6,12,20,30,\ldots\}$,  is  equal to $p(p-1)${.  Indeed}  for $g=p-2$,   $C_{2p, [p^2]}^{(p-2)}= \binom{p-1}{ p-2} \binom{p}{ p-1}=p(p-1)$.

\section{Genus-dependent Fa\`a di Bruno coefficients $C_{n, [\alpha]}^{(g)}$. Part II. A compilation of partial results}
\label{sect6}

\subsection{About types $[\alpha] = [p^k]$ for given $p=2,3,4,\ldots$ as a function of $k$} 
 
The results for $C_{n, [\alpha]}^{(g)}$ when $[\alpha]={[p^k]}$ are gathered in Table \ref{tablepwedgek}.  
For given $p$ and $k$ the values are listed vertically (downward) 
 according to the genus $g$, for $g=0,1,2,\ldots$
Here $n=k p$, (\ie $k$ parts of length $p$);  we have $g=(k(p-1)+1-f)/2$, therefore  $g \leq k(p-1)/2$.
These values have been obtained by an explicit determination of the genus for computer generated set partitions, or obtained from general theorems. 

\begin{table}[ht]
\begin{minipage}{\textwidth}
\begin{tabular}{c}
\resizebox{0.97\textwidth}{!}{
\begin{tabular}{C|C|C|C|C}
 & \text{$k$=1} & \text{$k$=2} & \text{$k$=3} & \text{$k$=4} \\
\hline
 \text{$p$=2} & 1 & 2,1 & 5,10 & 14,70,21 \\
 \text{$p$=3} & 1 & 3,6,1 & 12,102,144,22 & 55,1212,6046,7163,924 \\
 \text{$p$=4} & 1 & 4,18,12,1 & 22,432,2007,2604,710 & 140,7236,108090,592824,1180364,688270,50701 \\
 \text{$p$=5} & 1 & 5,40,60,20,1 & 35,1240,12060,41820,51565,18540,866 & 285,26800,809960,\ldots  \\
 \text{$p$=6} & 1 & 6,75,200,150,30,1 & 51,2850,47475,316700,905415,1076238,462375,47752 & 506,75450,3837575,\ldots  \\
 \text{$p$=7} & 1 & 7,126,525,700,315,42,1 & 70,5670,144270,\ldots  & 819,177660,13656006,\ldots 
\end{tabular}}
\\ 
{}
\\
\resizebox{0.97\textwidth}{!}{
\begin{tabular}{C|C|C|C}
 & \text{$k$=5} & \text{$k$=6} & \text{$k$=7} \\
\hline
 \text{$p$=2} & 42,420,483 & 132,2310,6468,1485 & 429,12012,66066,56628 \\
 \text{$p$=3} & 273,12330,149674,576660,593303,69160 & 1428,114888,2771028,\ldots  & 7752,1011486,42679084,\ldots  \\
 \text{$p$=4} & 969, 103680, 3588318,\ldots  & 7084,1359882,90800208,\ldots  
 & 53820, 16846704, 1929948363, \ldots  \\
 \text{$p$=5} & 2530,495200,34034480,\ldots  & 23751,8373000,1097464620,\ldots  & 231880,133685440,29830376800,\ldots  \\
 \text{$p$=6} & 5481,1707000,195525750 \ldots  & 62832,35331000,7670848500,\ldots  & 749398,690413850,254134018600,\ldots  \\
 \text{$p$=7} & 10472,4755870,818352528,\ldots  & 141778,116450460,37838531178,\ldots  & 1997688,2691733464,1479039054696,\ldots  
\end{tabular}}
\end{tabular}
\end{minipage}
\caption{Table of coefficients $C_{n, [\alpha]}^{(g)}$ for $n=k\, p$,\;  $[\alpha]={[p^k]}$.}
\label{tablepwedgek}
\end{table}

In this section we describe some generic features of the sequences that are obtained for increasing values of $p$, for various choices of $[\alpha] = [p^k]$.
The situation where $[\alpha] = [p^2]$, which is a particular case of the type  $[\alpha] = [p,n-p]$ considered in Section \ref{2part-partitions}---a``solved case''---(see (\ref{conjpn-p}) or  (\ref{conjpn-p-alt})), was already discussed at the end of   \ref{2part-partitions}.

\bigskip 
We recall from (\ref{faadiBruno}) that: 
\begin{itemize}
\renewcommand\labelitemi{--}
\item the $k$-th row's sum (over $g$) in the Table of $[p^k]$ is $(pk)!/(k!(p!)^k)$.  See  OEIS sequences \seqnum{A025035},
 \seqnum{A025036}, \seqnum{A025037}, \seqnum{A025038}, \seqnum{A025039},  for $p=3,\ldots,7$. 
\item and from (\ref{genus0}) that
 $C_{p.k, [p^k]}^{(0)} =\inv{pk+1} \binom{pk+1}{ k}={\inv{ (p-1)k+1}} \binom{pk}{ k}$.
 \end{itemize}

\subsubsection{Genus $g=1$: Observations and conjectures}

For $k=2$, it follows from Proposition~\ref{oldprop1} that
\begin{equation} C_{2p, [p^{2}]}^{(1)}=\frac{p(p-1)^2}{ 2}\end{equation}
which are the ``pyramidal pentagonal numbers'', $\{0, 1, 6, 18, 40, 75, 126, 196, 288, 405, \ldots\}$.  See OEIS sequences \seqnum{A002411}.
Furthermore, we observe that these pyramidal pentagonal numbers factorize the coefficients $C_{p k, [p^{k}]}^{(1)}$
\begin{equation} C_{p k, [p^{k}]}^{(1)} \? C_{2p, [p^{2}]}^{(1)} \phi(p, k)\end{equation}
\begin{equation}  \phi(p,k)\ =\ \begin{matrix}  0 & 1 & 10 & 70 &420 &\cdots \\
 				    0 & 1 & 17 & 202 & 2055& \cdots \\
				    0 & 1 & 24 & 402 &5760 &\cdots \\
				    0 & 1 & 31 & 670 &12380 &\cdots \\
				    0 & 1 & 38 & 1006 &22760& \cdots \\
				    0 & 1 & 45 & 1410 & 37745 & \cdots \\
				    \vdots &&&  
				    \end{matrix}   \qquad 2\le  p\le 7\end{equation}
in which the third column is an arithmetic series $7p-4$, the fourth $ 34 p^2 -38 +10$, etc.

In other words, the above tables are  compatible with the following
expressions\footnote{The expression for $k=3$ is a particular case
of a formula that will be discussed in Section~\ref{threeparts}.}: 
\begin{align} \nonumber
C_{3p,[p^3]}^{(1)}&=\frac{p(p-1)^2}{ 2}(7p-4),  \\ \nonumber
C_{4p,[p^4]}^{(1)}&{\?}\frac{p(p-1)^2}{ 2} (34p^2-38p +10), \\ 
\label{3parts}
C_{5p,[p^5]}^{(1)}&{\?}\frac{p(p-1)^2}{ 2} \, \frac{5}{6}(169 p^3 -279 p^2 +146 p -24),  \\ \nonumber
C_{6p,[p^6]}^{(1)}&{\?}\frac{p(p-1)^2}{ 2} \, (533 p^4 -1160 p^3 +\frac{1813}{2} p^2 -\frac{599}{2} p+35)\,  \\ \nonumber
C_{7p,[p^7]}^{(1)}&{\?}\frac{p(p-1)^2}{ 2}  \frac{7}{120} (32621 p^5 - 87970 p^4 + 91335 p^3 - 45410 p^2 + 10744 p-960  ).
\end{align}  
The constant terms in the polynomial $\phi$ appear to be (up to a sign $(-1)^k$) the ``tetrahedral (or triangular pyramidal) numbers'': $ k(k^2-1)/6$,  OEIS sequence \seqnum{A000292}.

\subsubsection{Genus $g=2$: Observations and conjectures}
\begin{align} \nonumber
C_{3p,[p^3]}^{(2)} &\? \inv{8}p ( p-1)^2 (p-2)  (27 - 55 p + 26 p^2), \\
C_{4p,[p^4]}^{(2)} &\? \inv{6}p {(p-1)^2}  (287 -1248 p +1908 p^2 - 1218 p^3 + 274 p^4),\\ \nonumber
C_{5p,[p^5]}^{(2)} &\? \inv{144} p ( p-1)^2 (-30576 + 194318 p - 467213 p^2 + 532986 p^3 - 288895 p^4 + 59500 p^5).
\end{align}
\noindent In each case, the conjecture has been tested on at least two more values than those used in the extrapolation.
\bigskip

 \subsection{Cases $[\alpha] =  [n-p-q, p,q]$.
 Partition of $n$ into $k=3$ parts}
 \label{threeparts}
 All the data that have been collected in that case,  when the genus is 0 or 1, 
  (see the Tables in the Appendix),  are given below.
\begin{align}
\label{threeparts0} 
C^{(0)}_{n,[n-p-q,p,q]} &=  n(n-1)\\
 \label{threeparts1} 
C^{(1)}_{n,[n-p-q,p,q]} &{=}  \frac{n}{2}(-5 (n - 1)^2 + 3 (p^2 + q^2 + r^2 - 1) +  6 p q r + (r+p) ( r + q) (p + q) ),
\end{align}
where in the last expression, the symmetry in the exchange of $p$, $q$ and $r:=n-p-q$ is manifest.
 These  expressions have to be multiplied by  $ \inv{2}$ if two of the three integers $p,q,n-p-q$ are equal, 
and by $\inv{6}$ if  $p=q=n-p-q$.

Equation (\ref{threeparts0}) for genus 0 is an immediate consequence of Kreweras' formula, (\ref{genus0}).
Equation (\ref{threeparts1}) was presented as a conjecture in a   first version of 
 this paper  but  it has been subsequently proved by Hock. See our comments in the acknowledgments section.
 
\section{Tables}

The following tables provide the
coefficients $C_{n,[\alpha]}^{(g)}$ for  $2\le n \le 15$ and $\alpha_1=0$.
Partitions without singletons are ordered by
increasing number of parts $|\alpha|$ 
and then by lexicographic order on $[\alpha]$. 
\begin{table}[H]
\begin{minipage}{0.45\textwidth}
\scalebox{0.6}{
\begin{tabular}{C|C|CCCCCCCC} 
& \hfill  g & &0&1&2&3&4&5\\ 
& [\alpha]&     &&&&&&\\
\hline
{n=2} & [2] &&1& 0& 0& 0& 0& 0\\
\hline
{n=3}& [3] &&1& 0& 0& 0& 0& 0\\
\hline
{n=4}&[4] &&1& 0& 0& 0& 0& 0\\[3pt]
& [2^2] &&2& 1& 0& 0& 0& 0\\
\hline
{n=5}&[5] &&1& 0& 0& 0& 0& 0\\[3pt]
& [2, 3] &&5& 5& 0& 0& 0& 0\\
\hline
{n=6}&[6] &&1& 0& 0& 0& 0& 0\\[3pt]
& [2,4] &&6& 9& 0& 0& 0& 0\\
& [3^2] &&3& 6& 1& 0& 0& 0\\[3pt]
& [2^3] &&5& 10& 0& 0& 0& 0\\
\hline
{n=7}&[7] &&1& 0& 0& 0& 0& 0\\[3pt]
& [2,5] &&7& 14& 0& 0& 0& 0\\
& [3, 4] &&7& 21& 7& 0& 0& 0\\[3pt]
& [2^2,3] &&21& 70& 14& 0& 0& 0\\
\hline
{n=8}&[8] &&1& 0& 0& 0& 0& 0\\[3pt]
& [2,6] &&8& 20& 0& 0& 0& 0\\
& [3,5] &&8& 32& 16& 0& 0& 0\\
& [4^2] &&4& 18& 12& 1& 0& 0\\[3pt]
& [2^ 2,4] &&28& 128& 54& 0& 0& 0\\
& [2,3^2] &&28& 152& 100& 0& 0& 0\\[3pt]
& [2^4] &&14& 70& 21& 0& 0& 0\\
\hline
{n=9}&[9] &&1& 0& 0& 0& 0& 0\\[3pt]
& [2,7] &&9& 27& 0& 0& 0& 0\\
& [3,6] &&9& 45& 30& 0& 0& 0\\
& [4,5] &&9& 54& 54& 9& 0& 0\\[3pt]
& [2^2,5] &&36& 207& 135& 0& 0& 0\\
& [2, 3, 4] &&72& 531& 603& 54& 0& 0\\
& [3^3] &&12& 102& 144& 22& 0& 0\\[3pt]
& [3, 2^3] &&84& 630& 546& 0& 0& 0\\
\end{tabular}}
\end{minipage}
\begin{minipage}{0.45\textwidth}
\scalebox{0.6}{  
  \begin{tabular}{C|C|CCCCCCCC} 
&  \hfill  g &0&1&2&3&4 & 5 \\ 
& [\alpha]&     &&&&&&\\ 
\hline
{n=10}&[10] &1& 0& 0& 0& 0& 0\\[3pt]
& [2,8] &10& 35& 0& 0& 0& 0\\
& [3,7] &10& 60& 50& 0& 0& 0\\
& [4,6] &10& 75& 100& 25& 0& 0\\
& [5^2] &5& 40& 60& 20& 1& 0\\[3pt]

& [2^2,6] &45& 310& 275& 0& 0& 0\\ 
& [2, 3, 5] &90& 830& 1340& 260& 0& 0\\ 
& [2,4^2] &45& 450& 840& 240& 0& 0\\ 
& [3^2,4] &45& 510& 1115& 430& 0& 0\\[3pt]

& [2^3,4] &120& 1165& 1685& 180& 0& 0\\ 
& [2^2, 3^ 2] &180& 1985& 3565& 570& 0& 0\\[3pt] 

& [2^5] &42& 420& 483& 0& 0& 0\\
\hline
{n=11}&[11] &1&0&0&0&0 &0\\[3pt]
& [2,9] & 11 & 44 & 0 & 0 & 0 & 0 \\
 & [3,8]&11 & 77 & 77 & 0 & 0 & 0 \\
& [4,7] & 11 & 99 & 165 & 55 & 0 & 0 \\
&[5,6] & 11 & 110 & 220 & 110 & 11 & 0 \\[3pt]

&[2^2,7]  &55 & 440 & 495 & 0 & 0 & 0 \\
& [2,3,6] & 110 & 1210 & 2530 & 770 & 0 & 0 \\
&[2,4,5] & 110 & 1375 & 3564 & 1793 & 88 & 0 \\
& [3^2,5] & 55 & 770 & 2277 & 1452 & 66 & 0 \\
&[3,4^2]& 55 & 825 & 2684 & 2035 & 176 & 0 \\[3pt]

&[2^3,5] & 165 & 1936 & 3905 & 924 & 0 & 0 \\
&[2^2,3,4]& 495 & 7007 & 19085 & 8063 & 0 & 0 \\
&[2,3^3] &165 & 2607 & 8195 & 4433 & 0 & 0 \\[3pt]

& [3,2^4] & 330 & 4620 & 10395 & 1980 & 0 & 0 \\
\end{tabular}}
\end{minipage}
\end{table}

\newpage

\begin{table}[H]
\begin{minipage}{0.45\textwidth}
\scalebox{0.6}{
\begin{tabular}{C|C|CCCCCCCC}
& \hfill  g  &0&1&2&3&4&5\\
& [\alpha]&     &&&&&&\\
\hline
{n=12}&[12] &1&0&0&0&0 &0\\[3pt]
&[2,10]  &12 & 54 & 0 & 0 & 0 & 0 \\
&[3,9] &12 & 96 & 112 & 0 & 0 & 0 \\
&[4,8] &12 & 126 & 252 & 105 & 0 & 0 \\
&[5,7]& 12 & 144 & 360 & 240 & 36 & 0 \\
 &[6^2]&6 & 75 & 200 & 150 & 30 & 1 \\[3pt]
 
&[2^2,8]&66 & 600 & 819 & 0 & 0 & 0 \\
&[2,3,7]& 132 & 1680 & 4308 & 1800 & 0 & 0 \\
&[2,4,6]& 132 & 1968 & 6510 & 4740 & 510 & 0 \\
 &[2,5^2]& 66 & 1032 & 3672 & 3072 & 474 & 0 \\
 &[3^2,6]& 66 & 1092 & 4062 & 3640 & 380 & 0 \\
  &[3,4,5]& 132 & 2436 & 10500 & 12084 & 2568 & 0 \\
  &[4^3]& 22 & 432 & 2007 & 2604 & 710 & 0 \\[3pt]

 &[2^3,6]& 220 & 3000 & 7730 & 2910 & 0 & 0 \\
  &[2^2,3,5]& 660 & 11232 & 40716 & 28968 & 1584 & 0 \\
  &[2^2, 4^2]& 330 & 5988 & 23877 & 20097 & 1683 & 0 \\
  &[2,3^2,4]& 660 & 13218 & 59076 & 59442 & 6204 & 0 \\
  &[3^4]& 55 & 1212 & 6046 & 7163 & 924 & 0 \\[3pt]

&[2^4,4]& 495 & 8616 & 28590 & 14274 & 0 & 0 \\
&[2^3,3^2]& 990 & 19104 & 73050 & 45456 & 0 & 0 \\[3pt]

&[2^6]& 132 & 2310 & 6468 & 1485 & 0 & 0 \\
\hline
{n=13}&[13] &1&0&0&0&0&0\\[3pt]
&[2,11]& 13 & 65 & 0 & 0 & 0 & 0 \\
&[3,10] &13 & 117 & 156 & 0 & 0 & 0 \\
&[4,9] &13 & 156 & 364 & 182 & 0 & 0 \\
&[5,8] &13 & 182 & 546 & 455 & 91 & 0 \\
 &[6,7]&13 & 195 & 650 & 650 & 195 & 13 \\[3pt]

 &[2^2,9]&78 & 793 & 1274 & 0 & 0 & 0 \\
 &[2,3,8] &156 & 2249 & 6825 & 3640 & 0 & 0 \\
 &[2,4,7] &156 & 2691 & 10803 & 10335 & 1755 & 0 \\
  &[2,5,6]&156 & 2912 & 13130 & 15470 & 4238 & 130 \\
 &[3^2,7]&78 & 1482 & 6630 & 7670 & 1300 & 0 \\
&[3,4,6] &156 & 3393 & 18161 & 28145 & 9945 & 260 \\
&[3,5^2] &78 & 1768 & 9984 & 16796 & 7046 & 364 \\
&[4^2,5]& 78 & 1872 & 11271 & 20904 & 10322 & 598 \\[3pt]

 &[2^3,7]& 286 & 4420 & 13819 & 7215 & 0 & 0 \\
 &[2^2,3,6]& 858 & 16900 & 76037 & 76765 & 9620 & 0 \\
&[2^2,4,5]& 858 & 18720 & 97539 & 125606 & 27547 & 0 \\
&[2,3^2,5]& 858 & 20488 & 117806 & 175916 & 45292 & 0 \\
&[2,3,4^2]& 858 & 21684 & 133887 & 222781 & 71240 & 0 \\
&[3^3,4]& 286 & 7878 & 53404 & 100997 & 37635 & 0 \\[3pt]

&[2^4,5]& 715 & 14612 & 63323 & 53053 & 3432 & 0 \\
&[2^3,3,4] &2860 & 68172 & 363610 & 419068 & 47190 & 0 \\
&[2^2,3^3] &1430 & 37336 & 221715 & 298649 & 41470 & 0 \\[3pt]
 
 &[3,2^5]&1287 & 30030 & 138138 & 100815 & 0 & 0 \\
 \end{tabular}}
 \end{minipage}
\quad
\begin{minipage}{0.50\textwidth}     
\scalebox{0.60}{  
 \begin{tabular}{C|C|CCCCCCCC} 
&  \hfill  g & &0&1&2&3&4&5 &6\\ 
& [\alpha]&     &&&&&&\\ 
\hline
{n=14}&[14] &&1&0&0&0&0&0&0\\[3pt]
&[2,12] &&14&77&0&0&0&0&0\\
&[3,11] &&14&140&210&0&0&0&0\\
&[4,10] &&14&189&504&294&0&0&0\\ 
&[5, 9] &&14&224&784&784&196&0&0\\ 
&[6, 8] &&14&245&980&1225&490&49&0\\ 
&[7^2] &&7&126&525&700&315&42&1\\[3pt] 

 &[2^2,10] &&91&1022&1890&0&0&0&0\\
&[2, 3, 9] &&182&2926&10248&6664&0&0&0\\
&[2, 4, 8] &&182&3556&16758&19894&4655&0&0\\
&[2,5,7] &&182&3934&21420&32620&13034&882&0\\
&[2,6^2] &&91&2030&11550&18900&8645&826&0\\
&[3^2,8] &&91&1946&10157&14406&3430&0&0\\
&[3,4,7] &&182&4536&28938&56434&28280&1750&0\\
&[3,5,6] &&182&4858&33824&75040&48174&6090&0\\
&[4^2,6] &&91&2562&18893&45654&33285&4620&0\\
&[4,5^2] &&91&2660&20496&52710&42679&7490&0\\[3pt]

&[2^3,8] &&364&6265&22981&15435&0&0&0\\ 
&[2^2,3,7] &&1092&24290&129948&170380&34650&0&0\\
&[2^2,4,6] &&1092&27587&176351&307020&114940&3640&0\\ 
&[ 2^2,5^2] &&546&14343&96726&182756&80094&3913&0\\ 
&[2, 3^2, 6] &&1092&29988&209510&415870&178920&5460&0\\ 
&[2,3,4,5] &&2184&65674&513450&1200738&696794&43680&0\\ 
&[2, 4^3] &&364&11529&95613&243180&162099&12740&0\\ 
&[ 3^3,5] &&364&11844&100660&262696&173348&11648&0\\ 
&[3^2, 4^2] &&546&18690&168273&475195&356951&31395&0\\[3pt]

&[2^4,6] &&1001&23240&123214&146020&21840&0&0\\
&[ 2^3,3,5] &&4004& 111608& 754614& 1286642& 365652&0&0\\
&[2^3, 4^2] &&2002&58912&427777&809823&278061&0&0\\ 
&[2^2,  3^2,4] &&6006&191968&1525454&3268552&1314320&0 &0\\
&[2,3^4] &&1001&34692&301070&725193&339444&0&0\\[3pt]

&[2^5,4] &&2002&56378&354263&473242&60060&0&0\\ 
&[ 2^4,3^2] &&5005&153832 & 1075165& 1663893& 255255 &0&0\\[3pt]

&[2^7] &&429&12012&66066&56628&0&0&0\\
 \end{tabular}}
 \end{minipage}
  \end{table}

\newpage
 
 \begin{table}[H]
\begin{center}  
\scalebox{0.65}{             
  \begin{tabular}{C|C|CCCCCCCC} 
&  \hfill  g & &0&1&2&3&4&5 &6\\ 
& [\alpha]&     &&&&&&\\ 
\hline
{n=15}& [15] && 1 & 0 & 0 & 0 & 0 & 0 & 0 \\[3pt]
&[2, 13] && 15 & 90 & 0 & 0 & 0 & 0 & 0  \\
&[3, 12] && 15 & 165 & 275 & 0 & 0  & 0  & 0 \\
&[4, 11] && 15 & 225 & 675 & 450 & 0 & 0 & 0  \\
&[5, 10] && 15 & 270 & 1080 & 1260 & 378  & 0  & 0 \\
&[6, 9] && 15 & 300 & 1400 & 2100 & 1050 & 140  & 0 \\
&[7, 8] && 15 & 315 & 1575 & 2625 & 1575 & 315 & 15 \\[3pt]

&[2^2, 11] && 105 & 1290 & 2700 & 0 & 0  & 0  & 0   \\
&[2, 3, 10] && 210 & 3720 & 14760 & 11340 & 0 & 0  & 0 \\
&[2, 4, 9] && 210 & 4575 & 24720 & 35070 & 10500  & 0  & 0  \\
&[2, 5, 8] && 210 & 5145 & 32760 & 61215 & 32340 & 3465  & 0  \\
&[2, 6, 7] && 210 & 5430 & 37200 & 78000 & 50550 & 8610 & 180 \\
&[3^2, 9] && 105 & 2490 & 14835 & 24920 & 7700   & 0  & 0 \\
&[3, 4, 8] && 210 & 5880 & 43470 & 102165 & 66675 & 6825  & 0 \\
&[3, 5, 7] && 210 & 6420 & 53010 & 146040 & 127290 & 26940 & 450 \\
&[3, 6^2] && 105 & 3300 & 28200 & 81500 & 77075 & 19380 & 650 \\
&[4^2, 7] && 105 & 3375 & 29385 & 87510 & 84975 & 19575 & 300 \\
&[4, 5, 6] && 210 & 7185 & 67260 & 221310 & 252090 & 79575 & 3000 \\
&[5^3] && 35 & 1240 & 12060 & 41820 & 51565 & 18540 & 866 \\[3pt]

&[2^3, 9] && 455 & 8610 & 36190 & 29820 & 0 & 0   & 0   \\
&[2^2, 3, 8] && 1365 & 33705 & 208320 & 336210 & 96075  & 0  & 0 \\
&[2^2, 4, 7] && 1365 & 38955 & 293670 & 644535 & 347175 & 25650  & 0 \\
&[2^2, 5, 6] && 1365 & 41580 & 342000 & 853500 & 579405 & 74040  & 0 \\
&[2, 3^2, 7] && 1365 & 42105 & 344595 & 853035 & 522450 & 38250  & 0 \\
&[2, 3, 4, 6] &&   2730 & 94290 & 886620 & 2671710 & 2296950 & 354000  & 0 \\
&[2, 3, 5^2] &&  1365 & 48825 & 479040 & 1530360 & 1449285 & 274905  & 0  \\
&[2, 4^2, 5] && 
           1365 & 51240 & 529425 & 1816410 & 1913445 & 417840  & 0  \\
&[3^3, 6] && 455 & 16905 & 171635 & 570805 & 550950 & 90650  & 0  \\
&[3^2, 4, 5] && 
           1365 & 55020 & 612705 & 2303715 & 2692020 & 641475   & 0  \\
&[3, 4^3] && 455 & 19215 & 224925 & 902400 & 1158840 & 321790  & 0 \\[3pt]

&[2^4, 7] && 1365 & 35250 & 219660 & 337050 & 82350  & 0   & 0 \\
&[2^3, 3, 6] && 
           5460 & 172350 & 1398740 & 3159450 & 1515700 & 54600  & 0  \\
&[2^3, 4, 5] && 
           5460 & 188025 & 1708035 & 4524870 & 2847420 & 185640  & 0 \\
&[2^ 2, 3^2, 5] && 
           8190 & 304335 & 2994180 & 8823360 & 6327465 & 461370  & 0  \\
&[2^2, 3, 4^2] && 
           8190 & 319605 & 3329865 & 10604790 & 8605395 & 780780  & 0  \\
&[2, 3^3, 4] && 
           5460 & 229365 & 2581215 & 9060870 & 8289600 & 854490  & 0 \\
&[3^5] && 273 & 12330 & 149674 & 576660 & 593303 & 69160 & 0  \\[3pt]

&[2^5, 5] && 3003 & 97140 & 761880 & 1493520 & 482292  & 0  & 0 \\
&[2^4 3, 4] && 15015 & 554250 & 5104260 & 12450900 & 5524200  & 0 & 0  \\
&[2^3, 3^3] && 10010 & 399660 & 4013730 & 10965465 & 5632135 & 0 & 0  \\[3pt]

&[2^6, 3] && 
             5005 & 180180 & 1471470 & 2622620 & 450450 & 0 & 0    \\
            \end{tabular}}
\end{center}
  \end{table}

\newpage
The following table can be obtained  either directly, by generating for each $n$ all partitions with a fixed number of parts, then calculating their genus, or, indirectly, by summing appropriate rows of the table of the genus-dependent Fa\`a di Bruno coefficients, while taking into account singletons (since the following is a table for $S_{n,k}^{(g)}$, not for $\widehat S_{n,k}^{(g)}$).
For each value of $n$ the number of parts, $k$  appears vertically, and the genus $g=0,1,\ldots$ increases in each row.
\begin{center}
\scalebox{0.55}{
\begin{tabular}{C|C|C|C|C|C|C|C|C}
\begin{tabular}{CC}
 1 &  1  \\
\end{tabular}
 &
\begin{tabular}{CC}
 1 &  1  \\
 2 &  1  \\
\end{tabular}
 &
\begin{tabular}{CC}
 1 &  1  \\
 2 &  3  \\
 3 &  1  \\
\end{tabular}
 &
\begin{tabular}{CC}
 1 &  1  \\
 2 &  6,1  \\
 3 &  6  \\
 4 &  1  \\
\end{tabular}
 &
\begin{tabular}{CC}
 1 &  1  \\
 2 &  10,5  \\
 3 &  20,5  \\
 4 &  10  \\
 5 &  1  \\
\end{tabular}
 &
\begin{tabular}{CC}
 1 &  1  \\
 2 &  15,15,1  \\
 3 &  50,40  \\
 4 &  50,15  \\
 5 &  15  \\
 6 &  1  \\
\end{tabular}
& 
\begin{tabular}{CC}
 1 &  1  \\
 2 &  21,35,7  \\
 3 &  105,175,21  \\
 4 &  175,175  \\
 5 &  105,35  \\
 6 &  21  \\
 7 &  1  \\
\end{tabular}
 &
\begin{tabular}{CC}
 1 &  1  \\
 2 &  28,70,28,1  \\
 3 &  196,560,210  \\
 4 &  490,1050,161  \\
 5 &  490,560  \\
 6 &  196,70  \\
 7 &  28  \\
 8 &  1  \\
\end{tabular}
 &
\begin{tabular}{CC}
 1 &  1  \\
 2 &  36,126,84,9  \\
 3 &  336,1470,1134,85  \\
 4 &  1176,4410,2184  \\
 5 &  1764,4410,777  \\
 6 &  1176,1470  \\
 7 &  336,126  \\
 8 &  36  \\
 9 &  1  \\
\end{tabular}
\end{tabular}
}
\end{center}

\bigskip

\begin{center}
\scalebox{0.6}{
\begin{tabular}{C|C|C}
 
\begin{tabular}{CC}
 1 &  1  \\
 2 &  45,210,210,45,1  \\
 3 &  540,3360,4410,1020  \\
 4 &  2520,14700,15330,1555  \\
 5 &  5292,23520,13713  \\
 6 &  5292,14700,2835  \\
 7 &  2520,3360  \\
 8 &  540,210  \\
 9 &  45  \\
 10 &  1  \\
\end{tabular}
 & 
\begin{tabular}{CC}
 1 &  1  \\
 2 &  55,330,462,165,11  \\
 3 &  825,6930,13860,6545,341  \\
 4 &  4950,41580,75075,24145  \\
 5 &  13860,97020,121275,14575  \\
 6 &  19404,97020,63063  \\
 7 &  13860,41580,8547  \\
 8 &  4950,6930  \\
 9 &  825,330  \\
 10 &  55  \\
 11 &  1  \\
\end{tabular}
 &
\begin{tabular}{CC}
 1 &  1  \\
 2 &  66,495,924,495,66,1  \\
 3 &  1210,13200,37422,29920,4774  \\
 4 &  9075,103950,289905,194150,14421  \\
 5 &  32670,332640,729960,284130  \\
 6 &  60984,485100,685608,91960  \\
 7 &  60984,332640,233772  \\
 8 &  32670,103950,22407  \\
 9 &  9075,13200  \\
 10 &  1210,495  \\
 11 &  66  \\
 12 &  1  \\
\end{tabular}
 \end{tabular}
}
\end{center}

\begin{center}
\scalebox{0.55}{
\begin{tabular}{C|C|C}
\begin{tabular}{CC}
 1 &  1  \\
 2 &  78,715,1716,1287,286,13  \\
 3 &  1716,23595,90090,109395,35464,1365  \\
 4 &  15730,235950,942942,1085370,252538  \\
 5 &  70785,990990,3396393,2797080,253253  \\
 6 &  169884,1981980,4972968,2196480  \\
 7 &  226512,1981980,3063060,443872  \\
 8 &  169884,990990,738738  \\
 9 &  70785,235950,52767  \\
 10 &  15730,23595  \\
 11 &  1716,715  \\
 12 &  78  \\
 13 &  1  \\
\end{tabular}
 & 
\begin{tabular}{CC}
 1 & 1 \\
 2 & 91,1001,3003,3003,1001,91,1 \\
 3 & 2366,40040,198198,340340,186186,21840 \\
 4 & 26026,495495,2690688,4759755,2288286,131495 \\
 5 & 143143,2642640,13096083,18768750,5424419 \\
 6 & 429429,6936930,27432405,25965940,2671669 \\
 7 & 736164,9513504,26342316,12737296 \\
 8 & 736164,6936930,11477466,1761760 \\
 9 & 429429,2642640,2063061 \\
 10 & 143143,495495,114114 \\
 11 & 26026,40040 \\
 12 & 2366,1001 \\
 13 & 91 \\
 14 & 1 \\
\end{tabular}
 \end{tabular}
 }
 \end{center}
\bigskip
\begin{center} 
\scalebox{0.55}{
\begin{tabular}{CC}
 1 &  1  \\
 2 &  105,1365,5005,6435,3003,455,15  \\
 3 &  3185,65065,405405,935935,775775,184275,5461  \\
 4 &  41405,975975,6921915,17482465,14369355,2564835  \\
 5 &  273273,6441435,43702659,97222125,58891833,4235595  \\
 6 &  1002001,21471450,123708585,205865660,68645577  \\
 7 &  2147145,38648610,169954785,177957780,20033013  \\
 8 &  2760615,38648610,115450335,59768280  \\
 9 &  2147145,21471450,37492455,6017440  \\
 10 &  1002001,6441435,5219214  \\
 11 &  273273,975975,230230  \\
 12 &  41405,65065  \\
 13 &  3185,1365  \\
 14 &  105  \\
 15 &  1  \\
\end{tabular}
}
\end{center}
\begin{center}
{\bf Table of $S_{n,k}^{(g)}$, from $n=1$ to $n=15$.}
\end{center}

The following table can be obtained  either directly, by generating, for each $n$, all partitions without singletons, with a fixed number of parts, then calculating their genus, or, indirectly, by summing appropriate rows of the table of the genus-dependent Fa\`a di Bruno coefficients (without singletons).
For each value of $n$ the number of parts, $k$  appears vertically, and the genus $g=0,1,\ldots$ increases in each row. Warning: the table starts at $n=2$.
\begin{center} 
\scalebox{0.6}{
\begin{tabular}{C|C|C|C|C|C|C|C}
n=2 & n=3 & n=4 & n=5 & n=6 & n=7 & n=8 & n=9 \\
\begin{tabular}{CC}
 1 &  1  \\
\end{tabular}
&
\begin{tabular}{CC}
 1 &  1  \\
\end{tabular}
&
\begin{tabular}{CC}
 1 &  1  \\
 2 &  2,1  \\
\end{tabular}
&
\begin{tabular}{CC}
 1 &  1  \\
 2 &  5,5  \\
\end{tabular}
&
\begin{tabular}{CC}
 1 &  1  \\
 2 &  9,15,1  \\
 3 &  5,10  \\
\end{tabular}
&
\begin{tabular}{CC}
 1 &  1  \\
 2 &  14,35,7  \\
 3 &  21,70,14  \\
\end{tabular}
&
\begin{tabular}{CC}
 1 &  1  \\
 2 &  20,70,28,1  \\
 3 &  56,280,154  \\
 4 &  14,70,21  \\
\end{tabular}
&
\begin{tabular}{CC}
 1 &  1  \\
 2 &  27,126,84,9  \\
 3 &  120,840,882,76  \\
 4 &  84,630,546  \\
\end{tabular}
\end{tabular}
}
\end{center} 

\begin{center} 
\scalebox{0.6}{
\begin{tabular}{C|C|C|C}
n=10 & n=11 & n=12   \\
\begin{tabular}{CC}
 1 &  1  \\
 2 &  35,210,210,45,1  \\
 3 &  225,2100,3570,930  \\
 4 &  300,3150,5250,750  \\
 5 &  42,420,483  \\
\end{tabular}
&
\begin{tabular}{CC}
 1 &  1  \\
 2 &  44,330,462,165,11  \\
 3 &  385,4620,11550,6050,330  \\
 4 &  825,11550,31185,13420  \\
 5 &  330,4620,10395,1980  \\
\end{tabular}
&
\begin{tabular}{CC}
 1 &  1  \\
 2 &  54,495,924,495,66,1  \\
 3 &  616,9240,31878,27940,4642  \\
 4 &  1925,34650,137445,118580,10395  \\
 5 &  1485,27720,101640,59730  \\
 6 &  132,2310,6468,1485  \\
\end{tabular}
\end{tabular}
}
\end{center} 

\bigskip
\begin{center} 
\scalebox{0.6}{
\begin{tabular}{C|C}
n=13 & n=14 \\
\begin{tabular}{CC}
 1 &  1  \\
 2 &  65,715,1716,1287,286,13  \\
 3 &  936,17160,78078,102960,34606,1352  \\
 4 &  4004,90090,492492,709280,191334  \\
 5 &  5005,120120,648648,770770,92092  \\
 6 &  1287,30030,138138,100815  \\
\end{tabular}
&
\begin{tabular}{CC}
 1 &  1  \\
 2 &  77,1001,3003,3003,1001,91,1  \\
 3 &  1365,30030,174174,322322,182182,21658  \\
 4 &  7644,210210,1513512,3273270,1797796,112476  \\
 5 &  14014,420420,3132129,6236230,2319317  \\
 6 &  7007,210210,1429428,2137135,315315  \\
 7 &  429,12012,66066,56628  \\
\end{tabular}
\end{tabular}
}
\end{center} 
\bigskip
\begin{center} 
\scalebox{0.6}{
\begin{tabular}{C}
n=15 \\
\begin{tabular}{CC}
 1 &  1  \\
 2 &  90,1365,5005,6435,3003,455,15  \\
 3 &  1925,50050,360360,890890,760760,182910,5446  \\
 4 &  13650,450450,4129125,12512500,11606595,2238600  \\
 5 &  34398,1261260,12381369,37087050,28261233,2406040  \\
 6 &  28028,1051050,9879870,24909885,11638627  \\
 7 &  5005,180180,1471470,2622620,450450  \\
\end{tabular}
\end{tabular}
}
\end{center} 
\begin{center}
{\bf Table of $\widehat S_{n,k}^{(g)}$, from $n=2$ to $n=15$.}
\end{center}

\section{Acknowledgments}
    After the present article was made available as a preprint (its version $1$ on arXiv) we received several comments from A. Hock suggesting that several partitions functions obtained in the present paper could be re-written by using changes of functions suggested by the formalism of topological recursion. This is in particular so for $g=1$ in the case of partitions into {three parts where expressions (\ref{threeparts1})} presented as conjectures in the first version of our paper could be proved by Hock by using this method{; his work has now been published  \cite{Hock}.}

    \smallskip
    It is also a pleasure to thank G. Chapuy, P. Di Francesco, and E. Garcia-Failde for stimulating discussions.

    \bigskip
    \hrule
    \bigskip

    \noindent 2020 {\it Mathematics Subject Classification}:
    Primary 05A18; Secondary 05A15,  15B52.

    \noindent \emph{Keywords: }  set-partition, genus-dependent enumeration.

    \bigskip
    \hrule
    \bigskip

    \noindent (Concerned with sequences
    \seqnum{A000108},
    \seqnum{A000110},
    \seqnum{A000292},
    \seqnum{A000296},
    \seqnum{A000581},
    \seqnum{A001263},
    \seqnum{A001287},
    \seqnum{A001700},
    \seqnum{A002411},
    \seqnum{A002450},
    \seqnum{A002802},
    \seqnum{A005043},
    \seqnum{A008277},
    \seqnum{A008299},
    \seqnum{A025035},
    \seqnum{A025036},
    \seqnum{A025037},
    \seqnum{A025038},
    \seqnum{A025039},
    \seqnum{A035319},
    \seqnum{A059260},
    \seqnum{A103371},
    \seqnum{A108263},
    \seqnum{A134991},
    \seqnum{A144431},
    \seqnum{A185259},
    \seqnum{A245551},
    \seqnum{A275514},
    \seqnum{A297178},
    \seqnum{A297179}, and
    \seqnum{A340556}.)

    \bigskip
    \hrule
    \bigskip

    \vspace*{+.1in}
    \noindent
    Received  June 29 2023;
    revised versions received  July 1 2023;
    January 15 2024; January 31 2024; February 2 2024; 
    February 4 2024.
    Published in {\it Journal of Integer Sequences},
    February 5 2024.

    \bigskip
    \hrule
    \bigskip

    \noindent
    Return to
   \href{https://cs.uwaterloo.ca/journals/JIS/}{Journal of Integer Sequences home page}.

    \vskip .1in

\begin{thebibliography}{99}


     \bibitem{BIPZ}  \'E. Br\'ezin, C. Itzykson, G. Parisi, and J.-B. Zuber,  Planar diagrams, {\it Comm. Math. Phys.} {\bf 59} (1978) 35--51.


      \bibitem{Caicedo-etal}  J. B. Caicedo, V. H. Moll, J. L. Ramirez, and D. Villamizar, 
    Extensions of set partitions and permutations,  {\it  Electron. J. Combin.}  
     {\bf 25} (2019), \#P2.20.
      
     
     \bibitem{Chapuy-th} G. Chapuy,  Combinatoire bijective des cartes de genre  sup\'erieur,  {\it PhD thesis}, \'Ecole Polytechnique, France,  2009.
     Available at 
     \url{https://pastel.archives-ouvertes.fr/pastel-00005289v1}.

    \bibitem{ChenReidys} R. X. F. Chen and C. M. Reidys, Narayana polynomials and some generalizations, arxiv preprint arXiv:1411.2530, 2014.
    Available at \url{http://arxiv.org/abs/1411.2530}.

     \bibitem{CoriH13} R. Cori and G. Hetyei, Counting genus one partitions and permutations,  {\it S\'em. Lothar. Combin.} {\bf 70} (2013), B70e.  Available at
    \url{http://arxiv.org/abs/1306.4628}.


    \bibitem{CoriH17} R. Cori and G. Hetyei, Counting partitions of a fixed genus, 
    {\it Electron. J. Combin.} {\bf 25} (4) (2018), \#P4.26.

      \bibitem{Cvitanovic} P. Cvitanovic, Planar perturbation expansion, {\it Phys. Lett. B} {\bf 99} (1981)  49--52.


     
     \bibitem{GoupilSch} A. Goupil and G. Schaeffer, Factoring $N$-cycles and counting maps of given genus, 
    {\it European J. Comb.} {\bf 19} (1998) 819--834.

     \bibitem{GKP} R. L. Graham, D. E. Knuth,  and O. Patashnik, {\it Concrete Mathematics: A Foundation for Computer Science}, 2nd ed., Addison-Wesley,  1994.

     \bibitem{HZ} J. Harer and D. Zagier, The Euler characteristic of the moduli space of curves, {\it Invent. Math.} {\bf 85} (1986) 457--485.
     
      \bibitem{Hock} A. Hock, Genus permutations and genus partitions,  arxiv preprint arXiv:2306.16237, 2023.
     Available at  \url{http://arxiv.org/abs/2306.16237}. 
     
     
     \bibitem{tH}G. 't Hooft,  A planar diagram theory for strong interactions, {\it Nucl. Phys. B} {\bf 72} (1974) 461--473.
     
     \bibitem{Jackson}  D. M. Jackson,  Counting cycles in permutations by group characters, with an application
    to a topological problem, {\it Trans. Amer. Math. Soc.}
    {\bf 299}  (1987) 785--801.

    \bibitem{Jacques} A. Jacques, Sur le genre d'une paire de substitutions, {\it C. R. Acad. Sci. Paris} {\bf 267} (1968) 625--627.


     \bibitem{Krew} G. Kreweras, Sur les partitions non crois\'ees d'un cycle, {\it Discrete Math.} {\bf 1} (1972) 333--350.



    \bibitem{LZ}  S. K. Lando and A. K. Zvonkin, {\it Graphs on Surfaces and  Their Applications}, {\it Encycl. of Math. Sci.}
     {\bf 141}, Springer, 2004.

    \bibitem{OEIS}  N. J. A. Sloane et al., {\it The On-Line Encyclopedia of Integer Sequences}, 2024.   Available at \url{https://oeis.org}.
     
    \bibitem{Speicher}   R. Speicher, Multiplicative functions on the lattice of non-crossing partitions and free convolution, {\it Math.  Ann.} {\bf 298} (1994) 611--628.

    \bibitem{WL1} T.  R. S. Walsh and A. B. Lehman,  Counting rooted maps by genus~I, {\it J. Combin. Theory Ser. B}
    {\bf 13} (1972) 192--218.

    \bibitem{WL2} T.  R. S. Walsh and A. B. Lehman,  Counting rooted maps by genus~II, {\it J. Combin. Theory Ser. B}
    {\bf 13} (1972) 122--141. 

    \bibitem{Yip} M. Yip, Genus one partitions, {\it PhD thesis}, University of Waterloo, 2006.  Available at 
    \url{https://uwspace.uwaterloo.ca/handle/10012/2933}.

    \bibitem{Z23}  J.-B. Zuber, Counting partitions by genus. I. Genus 0 to 2, 
    {\it Enumer. Comb. Appl. } {\bf 4} (2) (2024) \#S2R13.  Available at \url{https://doi.org/10.54550/ECA2024V4S2R13} and \url{http://arxiv.org/abs/2303.05875}.

    \end{thebibliography}
\end{document}